\documentclass[12pt]{article}
\usepackage[utf8]{inputenc}

\title{Convergence from Atomistic Model to Peierls-Nabarro Model for Dislocations in Bilayer System with Complex Lattice}
\author{Yahong Yang$^{a}$\thanks{E-mail address: \it\textbf{yyangct@connect.ust.hk}}, Tao Luo$^{b}$\thanks{E-mail address: \it\textbf{luotao41@sjtu.edu.cn}}, Yang Xiang$^{a}$\thanks{E-mail address: \it\textbf{maxiang@ust.hk} }\\$^{a}$\small\textit{Department of Mathematics, Hong Kong University of Science and Technology,}\\ \small\textit{Clear Water Bay, Kowloon, Hong Kong}\\$^{b}$\small\textit{School of Mathematical Sciences, Institute of Natural Sciences, }\\ \small\textit{MOE-LSC, and Qing Yuan Research Institute, Shanghai Jiao Tong University,}\\ \small\textit{ Shanghai, 200240, P.R. China} }

\usepackage{amsmath}
\usepackage{epsfig,amsthm}
\usepackage{amssymb,amsfonts}
\usepackage{dsfont}
\usepackage{subfigure}
\usepackage{indentfirst}
\usepackage{mathrsfs}
\usepackage{color}
\usepackage{graphicx}
\usepackage{zymacros}
\usepackage{hyperref}
\newtheorem{myre}{Remark}

\begin{document}

\maketitle
\begin{abstract}
	In this paper, we prove the convergence from the atomistic model to the Peierls--Nabarro (PN) model of two-dimensional bilayer system with complex lattice. We show that the displacement field of the dislocation solution of the PN model converges to the dislocation solution of the atomistic model with second-order accuracy. The consistency of PN model and the stability of atomistic model are essential in our proof. The
main idea of our approach is to use several low-degree polynomials  to approximate the energy due to atomistic interactions of different groups of atoms of the complex lattice.

\end{abstract}

\section{Introduction}

Dislocations are line defects in crystalline materials \cite{Hirth1982} and the dislocation theory is essential to understand the plastic deformation properties in crystalline materials.
Continuum theory of dislocations based on linear elasticity theory  works well outside the dislocation core region, which is a small region of a few lattice constants size around the dislocations; however, it breaks down inside the dislocation core where the atomic structure is heavily distorted.
 Atomistic models and Peierls-Nabarro models \cite{Hirth1982,nabarro1947dislocations,peierls1940size,vitek1968intrinsic} are widely applied to describe dislocation-core related properties. Although atomistic models are able to describe the details in the dislocation core region well, length and time scales of atomistic simulations are limited.
The   Peierls-Nabarro models \cite{nabarro1947dislocations,peierls1940size} are continuum models that incorporate the atomistic structure within the dislocation core region, and can be applied to simulations of larger length and time scales. The nonlinear elastic effect associated with the dislocation core is included in a  Peierls-Nabarro model by a nonlinear potential, i.e., the $\gamma$-surface \cite{vitek1968intrinsic}, that describes the
atomic interaction across the slip plane of a dislocation calculated by atomistic models or first principles calculations.

Although Peierls-Nabarro models with $\gamma$-surfaces   have been widely applied in the materials science problems associated with dislocations (e.g., \cite{Vitek1971-p493-508,Kaxiras1993-p3752-3755,schoeck1994generalized,Movchan1998-p373-396,
Lu2000-p3099-3108,Shen2004-p683-691,
Xiang2008-p1447-1460,Wei2008-p275-293,Dai2013-p1327-1337,Shen2014-p125-131,zhou2015van,dai2016structure}),
  mathematical understandings and rigorous analysis  on the accuracy of this class of models, i.e., convergence from the atomistic models, are still limited. The major challenge comes from the discontinuity in the displacement (i.e. disregistry) across the slip plane of the dislocation,  the nonlinear interaction across which is modeled on the continuum level by the $\gamma$-surface \cite{vitek1968intrinsic}. This situation is essentially different from the condition of the convergence from atomistic models to the Cauchy-Born rule \cite{Born1954-p-} (e.g. \cite{braides1999variational,Blanc2002-p341-381,friesecke2002validity,weinan2007cauchy}), in which smooth elastic displacement is assumed.
    El Hajj {\it et al.} and  Fino  {\it et al.}  \cite{el2009dislocation,fino2012peierls} showed  convergence using the
framework of viscosity solution from an atomistic model to the Peierls-Nabarro model on a square lattice based on a two-body potential with nearest neighbor interaction. However, the calculations of the $\gamma$-surfaces in real atomistic simulations are normally beyond the nearest neighbor atomic interactions.
Luo {\it et al.} \cite{luo2018atomistic} proved convergence from a full atomistic model to the
Peierls-Nabarro model with $\gamma$-surface for the dislocation in a bilayer system.
 In their convergence proof, each atom interacts with all other atoms  via a two-body interatomic potential whose effective interaction range is much larger than the nearest neighbor interaction, which is common in atomistic simulations. They also proved that the rate of such convergence is $O(\varepsilon^2)$ where $\varepsilon$ is the ratio of the length of the Burgers vector of the dislocation (the lattice constant) to the dislocation core width. These proofs are all for dislocations in crystals with simple lattices.
Silicon and other covalent-bonded crystals have  complex lattices (unions of simple lattices), for which two-body potentials do not work well \cite{stillinger1985computer} and the above proofs for convergence from atomistic models to Peierls-Nabarro models of dislocations  based on two-body potentials do not apply.

In this paper, we perform a rigorous  convergence proof from atomistic model of complex lattice to Peierls-Nabarro model with $\gamma$-surface for an inter-layer  dislocation in a  bilayer system.
We consider the bilayer graphene as the prototype. Each layer has a hexagonal lattice structure which can be considered as a union of two simple triangular lattices with a shift.  In our analysis, the interaction of atoms within the same layer is modeled by a three-body potential (e.g. the Stillinger-Weber potential \cite{stillinger1985computer} commonly used for silicon), and the inter-layer atomic interaction is described by a two-body potential (e.g. the van der Waals-like interactions described well by the Morse potentials in graphene, boron
nitride, and graphene/boron nitride bilayers \cite{zhou2015van}). We focus on a straight edge dislocation.

Our proof is a generalization of the convergence result for dislocations in simple lattice in Ref.~\cite{luo2018atomistic} inspired by the work of E and Ming \cite{weinan2007cauchy} for Cauchy-Born rule without defects, where consistency and stability properties of the atomistic and continuum models play key roles. Due to the complicated interactions on complex lattices under multi-body potentials, a major challenge in the convergence proof from atomistic model on a complex lattice to continuum model is how to construct a simple and accurate link between them. Here we propose a new approach to solve this problem. The main idea of our approach is to use several (here two) low-degree polynomials (here piecewise linear functions) to approximate the energy due to atomistic interactions among different groups of atoms on the complex lattice.
This energy based method is different from that based on Taylor expansions of displacement vectors in the force equations used in Ref.~\cite{weinan2007cauchy}.
 This new approach can be applied generally to continuum modeling based on atomistic structures with complex lattices.

Specifically, we follow the energy method in Ref.~\cite{luo2018atomistic} for the link (consistency and stability) between atomistic model and the Peierls-Nabarro model. Note that energy is an important quantity for the Peierls-Nabarro model of dislocations. The energies of the atomistic model and the Peierls-Nabarro model play a critical role in the proof of stability, in which second variations of the energies are needed.  For the convergence from a simple lattice,
the simple approximation of the energy by piecewise linear interpolation from the atomistic model works well for this purpose \cite{luo2018atomistic}.
However, for the bilayer complex lattice that we are considering, there  two different kinds of atoms, A and B, from two simple lattices, and four kinds of interactions by the three-body potential, i.e., AAA, AAB, ABB and BBB. It is tempting
 to use higher-degree interpolation polynomials for the energy of these complicated atomic interactions.
 However, the degrees of  suitable interpolation polynomials have to be very high, leading to oscillations in the interpolation energy.  Alternatively, we propose
a new way to construct continuum approximation of energy from the atomistic model of complex lattice (in Sec.~9).
The total energy on atomistic level is divided into two parts, for the atomistic interactions among two different groups of atoms on the complex lattice, respectively.
For each part, we establish a piecewise linear interpolation function to approximate it.
This enables simple  calculations of the stability of the two models in the convergence proof. Moreover, note that in the complex lattice, although the dislocation is straight in the continuum model, the locations and displacements in the atomistic model are not the same atom by atom along the dislocation, and we need to perform analysis in the two dimensional slip plane instead of the one-dimensional simplification adopted in the analysis for the simple lattice in Ref.~\cite{luo2018atomistic}.

The paper is organized as follows. We present the atomistic model of complex lattice and Peierls--Nabarro model for an edge dislocation in a bilayer system in Secs.~2 and 3, respectively. A dimensionless small parameter $\varepsilon$ is defined due to the weak interlayer interaction and nondimensionalization is performed in Sec.~4. This small parameter $\varepsilon$ is also the ratio of the lattice constant (the length of the Burgers vector of the dislocation) to the dislocation core width. In Sec.~5, we
collect the notations and assumptions. The main theorems including the wellposedness of the two models and convergence from the atomistic model to the Peierls-Nabarro model with $\gamma$-surface are presented in Sec.~6. Especially, the error of the convergence is shown to be $O(\varepsilon^2)$.  Rigorous proofs of these theorems are presented in Secs.~7-9. In particular, the two piecewise linear interpolation functions that  approximate two parts of the energy of atomistic interactions on the complex lattice are defined in Sec.~9. A proper boundary condition of atomistic model called Atomistic Dislocation Condition (ADC) is proposed to account for the displacement condition on the dislocation in the complex lattice. With this defined energy, the consistency of Peierls-Nabarro model and the stability of atomistic model are proved by comparing the Peierls-Nabarro model with the atomistic model directly.


\section{Atomistic Model of Complex Lattice}

 In this section, we present the atomistic model in a bilayer complex lattice system (e.g. bilayer graphene), from which the convergence to the Peierls-Nabarro model with $\gamma$-surface for an inter-layer dislocation is proved.

In the two-dimensional setting, atoms of the bilayer system are located in the planes
  $z=\pm \frac{1}{2}d$, where $d$ is the distance between two layers.

In a plane perpendicular to the $z$ axis, a simple lattice takes the form $$
L\left(\boldsymbol{e}_{i}, \boldsymbol{o}\right)=\left\{\boldsymbol{x} \mid \boldsymbol{x}=\sum_{i=1}^{2} v^{i} \boldsymbol{e}_{i}+\boldsymbol{o},\quad v^{i}\in\mathbb{Z} \right\},
$$where $\{\ve_i\}_{i=1}^2$ is a basis of $\mathbb{R}^2$ and $\vo$ is the location of some atom of the lattice. A complex lattice is regarded as a union of  simple lattices:
 $$
L=L\left(\ve_{i}, \vo\right) \cup L\left(\ve_{i}, \vo+\vp_{1}\right) \cup \cdots L\left(\ve_{i}, \vo+\vp_{k}\right)
$$
for certain integer $k$, and $\vp_{1}, \dots, \vp_{k}$ are the shift vectors.

In particular, one  graphene layer consists two simple lattices (triangular lattices) $A$ and $B$ (see Fig.~\ref{graphene}), and can be represented by
\begin{equation}
\begin{aligned}
&L(\boldsymbol {e_{i}},\vo)=\left\{\boldsymbol{x} \mid v^1\boldsymbol {e_{1}}+v^2\boldsymbol {e_{2}}+\vo, v^{i}\in \mathbb{Z}\right\}&\text{ for $A$},\\
&L(\boldsymbol {e_{i}},\vp)=\left\{\boldsymbol{x} \mid v^1\boldsymbol {e_{1}}+v^2\boldsymbol {e_{2}}+\vo+\vp, v^{i}\in \mathbb{Z} \right\}&\text{ for $B$},\\
&{\textstyle \boldsymbol{e_{1}}=a(1,0), \boldsymbol{e_{2}}=a\left(\frac{1}{ 2}, \frac{\sqrt{3}}{ 2}\right), \vo=(0,0), \vp=a\left(\frac{1}{2},\frac{\sqrt{3}}{ 6}\right)},
\end{aligned}
\end{equation}
where $a$ is the lattice constant.

\begin{figure}[h]
	\centering
	\scalebox{0.35}{\includegraphics{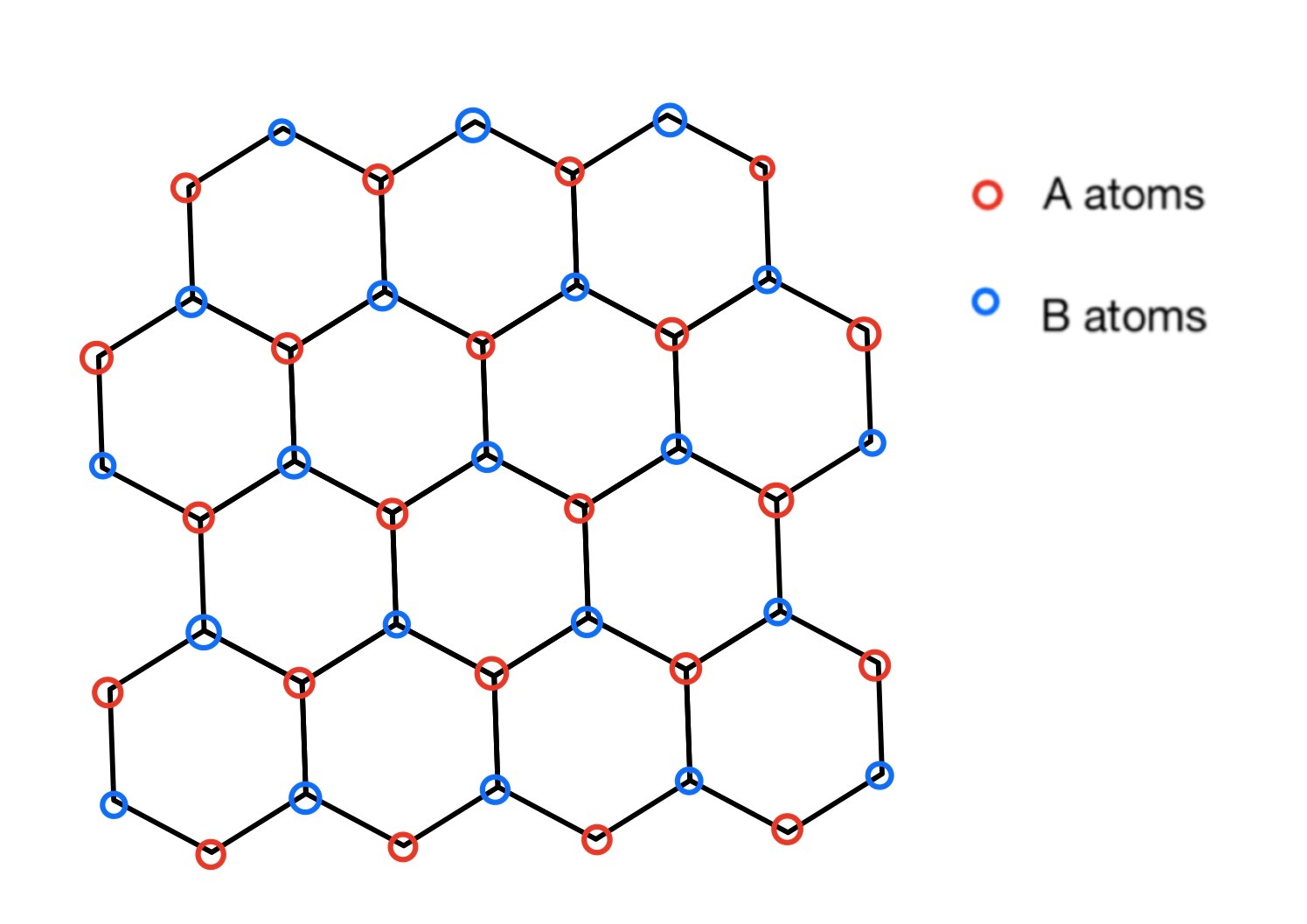}}
	\caption{The lattice of one graphene layer.}
	\label{graphene} 
\end{figure}

In this paper, we consider that a complex lattice arranged in such AB form. The layer located at $z=\frac{1}{2}d$ is
$L=L\left(\ve_{i}, \vo\right) \cup L\left(\ve_{i}, \vo+\vp\right)$,
and that located at $z=-\frac{1}{2}d$ is
$L=L\left(\ve_{i}, \vo+\vd\right) \cup L\left(\ve_{i}, \vo+\vp+\vd\right)$,
where $\vd=a\left(\frac{1}{2},-\frac{\sqrt{3}}{6}\right)$ is the vector between the centers of the two layers.

Suppose that there is an edge dislocation along the $y$ axis centered at $z=x=0$ with  Burgers vector $\boldsymbol{b}=(a,0,0)$. This dislocation generates a nonzero displacement field $\vu=\{\vu^{A+}_{\vs},\vu^{B+}_{\vs},\vu^{A-}_{\vs},\vu^{B-}_{\vs}\}$, where $\vs=(x_s,y_s)\in L$   is the location of the atom in the upper layer $z=\frac{1}{2}d$ (with superscript `$+$') or lower layer $z=-\frac{1}{2}d$ (with superscript `$-$'), `$A$' or `$B$' means that the atom is an $A$ atom or a $B$ atom.
  Here we assume that atoms in this bilayer system only have in-plane displacements, i.e., $\{\vu^{\kappa\pm}_{\vs}\}$ are two-dimensional vectors. Moreover, since the dislocation is a straight edge dislocation along the $y$ axis, the second coordinate of $\vu^{\kappa\pm}_{\vs}$ is $0$, and we denote $\vu^{\kappa\pm}_{\vs}=(u^{\kappa\pm}_{\vs},0)$. For this edge dislocation, the displacement field $u^{\kappa\pm}_{\vs}$ satisfies the following conditions:
\begin{equation}\label{eqn:aphi}
  \lim_{x_s\to+\infty}\left(u_{\vs}^{\kappa+}-u_{\vs}^{\kappa-}\right)=a, \
  \lim_{x_s\to-\infty}\left(u_{\vs}^{\kappa+}-u_{\vs}^{\kappa-}\right)=0, \
  u_{\vs}^{\kappa+}=-u_{\vs}^{\kappa-},
\end{equation}
where $\vs=(x_s,y_s)$ and $ \kappa=A$ or $B$.

The atomic interactions in the bilayer system can be divided into two parts: the intralayer interaction (covalent-bond interaction) and the inter-layer interaction (van der Waals-like interaction). For the intralayer interaction within each complex lattice layer, we focus on the three-body potential:
$$
V_{\rm tot}=\frac{1}{3!}\sum_{\substack{i,j,k,\\ i<j<k}}V_{3}\left(\boldsymbol{x}_{i}, \boldsymbol{x}_{j} \boldsymbol{x}_{k}\right).
$$
Recall that Stillinger-Weber potential \cite{stillinger1985computer} for the silicon crystal is the sum of interactions of two-body and three-body potentials, and the three-body potential is
$$
V_{3}\left(\boldsymbol{x}_{i}, \boldsymbol{x}_{j}, \boldsymbol{x}_{k}\right)=h\left(r_{i j}, r_{i k}, \theta_{j i k}\right),
$$
where
$$
h\left(r_{i j}, r_{i k}, \theta_{j i k}\right)=\lambda e^{\left[\gamma\left(r_{i j}-a\right)^{-1}+\gamma\left(r_{i k}-a\right)^{-1}\right]}\left(\cos \theta_{j i k}+1 / 3\right)^{2}
$$
for some parameters $\lambda$ and $\gamma$, $\vx_{i}, \vx_{j}, \vx_{k}$ are the three different atoms in the same layer, $\theta_{j i k}$ is angle between $\vx_{j}-\vx_{i}$ and $\vx_{k}-\vx_{i}$ and $r_{ij}=|\vx_i-\vx_j|,r_{ik}=|\vx_i-\vx_k|$. In fact,
this three-body potential $V_{3}\left(\boldsymbol{x}_{i}, \boldsymbol{x}_{j}, \boldsymbol{x}_{k}\right)$ depends only on the two vectors $\vs_1=\vx_j-\vx_i$ and $\vs_2=\vx_k-\vx_i$. For this reason, we assume that the three-body potential can be written in the form
\begin{equation}\label{eqn:3bd}
 V_3(\boldsymbol{x}_{i}, \boldsymbol{x}_{j}, \boldsymbol{x}_{k})=V\left(\frac{\vs_1}{a},\frac{\vs_2}{a}\right),
 \end{equation}
where $\vs_1=\vx_j-\vx_i$ and $\vs_2=\vx_k-\vx_i$.

Therefore, with displacement field $\vu=\{\vu^{A+}_{\vs},\vu^{B+}_{\vs}\}$, the intralayer interactions within the upper layer $z=+\frac{1}{2}d$ reads as
\begin{flalign}\label{eqn:elastic-atomistic}
E^{+}_{\text{a-el}}&=\sum_{\vx\in \mathbb{L}}\Bigg\{\frac{1}{6}\sum_{\langle \vs_1,\vs_2\rangle\in\mathbb{L}^*\times\mathbb{L}^* }\Bigg[V\left(\frac{\vs_1}{a}+\frac{\vu^{A+}_{\vx+\vs_1}-\vu^{A+}_{\vx}}{a},\frac{\vs_2}{a}
+\frac{\vu^{A+}_{\vx+\vs_2}-\vu^{A+}_{\vx}}{a}\right)\notag\\
&~~~+V\left(\frac{\vs_1}{a}+\frac{\vu^{B+}_{\vx+\vs_1}-\vu^{B+}_{\vx}}{a},\frac{\vs_2}{a}
+\frac{\vu^{B+}_{\vx+\vs_2}-\vu^{B+}_{\vx}}{a}\right)-2V\left(\frac{\vs_1}{a},\frac{\vs_2}{a}\right)\Bigg]\notag\\
&~~~+\frac{1}{2}\sum_{\langle\vs_1,\vs_2\rangle\in\mathbb{L}^*\times\mathbb{L}^*}
\Bigg[V\left(\frac{\vs_1}{a}+\frac{\vp}{a}
+\frac{\vu^{B+}_{\vx+\vs_1}-\vu^{A+}_{\vx}}{a},\frac{\vs_2}{a}+\frac{\vu^{A+}_{\vx+\vs_2}-\vu^{A+}_{\vx}}{a}\right)
\notag\\
&~~~+V\left(\frac{\vs_1}{a}-\frac{\vp}{a}+\frac{\vu^{A+}_{\vx+\vs_1}-u^{B+}_{\vx}}{a},\frac{\vs_2}{a}
+\frac{\vu^{B+}_{\vx+\vs_2}-\vu^{B+}_{\vx}}{a}\right)\notag\\
&~~~-V\left(\frac{\vs_1}{a}-\frac{\vp}{a},\frac{\vs_2}{a}\right)
-V\left(\frac{\vs_1}{a}+\frac{\vp}{a},\frac{\vs_2}{a}\right)\Bigg]\Bigg\},
\end{flalign}
where
\begin{equation}
    \mathbb{L}=L(\ve_i,\vo),~\mathbb{L}^*=L(\ve_i,\vo)\backslash\{\vzero\}.\label{condition}
\end{equation}
Note that in this formula, the summation with factor $\frac{1}{6}$ comes from $AAA$ and $BBB$ types three-body interactions, and the summation with factor $\frac{1}{2}$ comes from $AAB$ and $BBA$ types three-body interactions. In fact, since we use two side vectors as variables in the three-body interaction in Eq.~\eqref{eqn:3bd}, a type $AAA$ or $BBB$ three-body interaction is counted $3!$ times in the summations: each atom in such a three-body interaction can be used as the $\vx\in \mathbb{L}$ in the outer summation, and $\langle \vs_1,\vs_2\rangle$, $\langle \vs_2,\vs_1\rangle\in\mathbb{L}^*\times\mathbb{L}^*$ in the inner summation give the same three-body interaction.
Whereas for a type $AAB$ ($ABB$) interaction, we collect it by using an ``$A$" (``$B$") atom as the $\vx\in \mathbb{L}$ in the outer summation, and each of the two ``$A$" (``$B$") atoms in the $AAB$ ($ABB$) can be used as the $\vx\in \mathbb{L}$ in the outer summation, leading to a factor $2$ for this three-body interaction. These give the summation formula in Eq.~\eqref{eqn:elastic-atomistic}. Note that this method of organizing terms of intra-layer interactions on a complex lattice layer will be very convenient for the convergence proofs from atomistic model to the Peierls--Nabarro model in later sections.

We have similar formula for $E^{-}_{\text{a-el}}$, i.e., the intralayer interactions within the lower layer $z=-\frac{1}{2}d$ with displacement field $\vu=\{\vu^{A-}_{\vs},\vu^{B-}_{\vs}\}$, by using a different lattice
\begin{equation}
    \mathbb{L}=L(\ve_i,\vo+\vd),~\mathbb{L}^*=L(\ve_i,\vo+\vd)\backslash\{\vzero\}, \label{condition11}
\end{equation}
and replacing all the variables with notation ``$+$"  by those with notation ``$-$".
The total intra layer elastic energy is
\begin{equation}
E_{\text{a-el}}=E^{+}_{\text{a-el}}+E^{-}_{\text{a-el}}.
\end{equation}

Next, we describe the inter-layer interaction, which is a weak interaction, e.g. the van der Waals-like interaction. In \cite{zhou2015van},  Morse potential is used to describe the van der Waals-like interaction, which is a two-body potential:
$$
   V_\text{Morse}(r)=D_{e}\left(e^{-2 c\left(r-r_{e}\right)}-2 e^{-c\left(r-r_{e}\right)}\right),
$$
where  $r$ is the distance between two atoms, $r_{e}$ is the equilibrium bond distance, $D_{e}$ is the well depth,  and $c$ controls the width of the potential. We focus on such two-body potential for the inter-layer interaction:
\begin{equation}
V_\text{inter}\left(\vx^{+}_i, \vx^{-}_j\right)
=V_{\rm d}\left(\frac{|\vx^{+}_i- \vx^{-}_j|}{a}\right),
\end{equation}
where $\vx^{+}_i$ and $\vx^{-}_j$ are the locations of two atoms in the upper and lower layers, respectively, and the distance $|\vx^{+}_i- \vx^{-}_j|$ is the distance in three-dimensions.


Therefore, with displacement field $\vu=\{\vu^{A+}_{\vs},\vu^{B+}_{\vs},\vu^{A-}_{\vs},\vu^{B-}_{\vs}\}$ in the bilayer system, the inter-layer interaction energy in the bilayer system is
\begin{flalign}
E_{\text{a-mis}} &=\sum_{\vx\in\mathbb{L}}\sum_{\vs\in\mathbb{L}}\Bigg[V_{\rm d}\left(\frac{\vu^{A+}_{\vx+\vs}-\vu^{A-}_{\vx}}{a}
+\frac{\vs}{a}-\frac{\vd}{a}\right)-V_{\rm d}\left(\frac{\vs}{a}-\frac{\vd}{a}\right)\notag\\
&~~~+V_{\rm d}\left(\frac{\vu^{A+}_{\vx+\vs}-\vu^{B-}_{\vx}}{a}+\frac{\vs}{a}-\frac{\vd+\vp}{a}\right)
-V_{\rm d}\left(\frac{\vs}{a}-\frac{\vd+\vp}{a}\right)\notag\\
&~~~+V_{\rm d}\left(\frac{\vu^{B+}_{\vx+\vs}-\vu^{A-}_{\vx}}{a}+\frac{\vs}{a}-\frac{\vd-\vp}{a}\right)
-V_{\rm d}\left(\frac{\vs}{a}-\frac{\vd-\vp}{a}\right)\notag\\
&~~~+V_{\rm d}\left(\frac{\vu^{B+}_{\vx+\vs}-\vu^{B-}_{\vx}}{a}+\frac{\vs}{a}-\frac{\vd}{a}\right)
-V_{\rm d}\left(\frac{\vs}{a}-\frac{\vd}{a}\right)\Bigg],\label{eqn:a-mis}
\end{flalign}
where $\vd$ is recalled to be the vector between the centers of two layers. Here the four rows come from the $A^+A^-$, $A^+B^-$, $B^+A^-$, and $B^+B^-$ types interactions, respectively.

Note that for  simplicity of notations, we will omit $\mathbb{L}$ and $\mathbb{L}^*$ in the summations in later sections, e.g., $\sum_{\langle \vs_1,\vs_2\rangle}$ means $\sum_{\langle \vs_1,\vs_2\rangle\in\mathbb{L}^*\times\mathbb{L}^* }$.

\section{Peierls--Nabarro Model with $\gamma$-surface}

In the Peierls-Nabarro (PN) model, we may consider an edge dislocation with Burgers vector $\boldsymbol{b}=(a,0,0)$  centered at $x=z=0$, i.e., along the $y$ axis. The  locations of atoms are $\Gamma_{\text{PN}}=\{(\vx+\vu^{\pm},\pm\frac{1}{2}d),~\vx\in{\mathbb{R}^2} \}$. In the classical PN model, the disregistry (relative displacement) $\vphi(x)$ between the two layers is defined as
\begin{flalign}\label{eqn:phi1}
\vphi(\vx)=\vu^{+}(\vx)-\vu^{-}(\vx),~\vx\in{\mathbb{R}^2},
\end{flalign}
 and  satisfies the boundary conditions
 \begin{flalign}\label{eqn:phi2}
 \lim_{x\to -\infty}\vphi(x,y)=\vzero,\quad\lim_{x\to +\infty}\vphi(x,y)=\left(a,0\right).
 \end{flalign}
 In addition, the dislocation is centered at $x=z=0$, which means
 \begin{flalign}\label{eqn:phi3}
 \vphi(0,y)=\left(\frac{1}{2}a,0\right).
 \end{flalign}

For the straight edge dislocation being considered,  we have $\vu^\pm=(u^\pm,0)$ and $\vphi=(\phi,0)$.   In the classical PN model with $\gamma$-surface, the total energy of a bilayer graphene system is divided into two parts, the misfit energy (due to the interaction of two atoms from different layers) and the elastic energy (due to the interaction from the same layer).

In bilayer graphene system with $\gamma$-surface, the misfit energy is regarded as the interactions across the slip plane. Therefore we have the following formula by the atomistic model:
\begin{equation}\label{eqn:misfit00}
E_{\text{mis}}[\vphi]=\int_{\mathbb{R}^2}\gamma(\vphi(x,y))\,\D x \,\D y,
\end{equation}
where the density of this misfit energy $\gamma(\vphi)$ is the $\gamma$-surface~\cite{vitek1968intrinsic} which is defined as the energy increment per unit area. By the definition, the $\gamma$-surface can be calculated in terms of the atomistic model by \begin{align}\gamma(\vphi)&=\frac{2}{\sqrt{3}a^2}\sum_{\vs}\Bigg[V_{\rm d}\left(\frac{\vphi}{a}+\frac{\vs}{a}-\frac{\vd}{a}\right)-V_{\rm d}\left(\frac{\vs}{a}-\frac{\vd}{a}\right)\notag\\&~~~+V_{\rm d}\left(\frac{\vphi}{a}+\frac{\vs}{a}-\frac{\vd+\vp}{a}\right)-V_{\rm d}\left(\frac{\vs}{a}-\frac{\vd+\vp}{a}\right)\notag\\&~~~+V_{\rm d}\left(\frac{\vphi}{a}+\frac{\vs}{a}-\frac{\vd-\vp}{a}\right)-V_{\rm d}\left(\frac{\vs}{a}-\frac{\vd-\vp}{a}\right)\notag\\&~~~+V_{\rm d}\left(\frac{\vphi}{a}+\frac{\vs}{a}-\frac{\vd}{a}\right)-V_{\rm d}\left(\frac{\vs}{a}-\frac{\vd}{a}\right)\Bigg],\end{align}

The elastic energy in the PN model comes from the intralayer elastic interaction in the two layers. The elastic energy in the upper layer $E^+_\text{elas}$ is
\begin{align}
    E^+_\text{elas}&=\int_{\mathbb{R}^2}W(\nabla u^+)+W(\nabla u^-)\,\D x\,\D y\notag\\&=\int_{\mathbb{R}^2}\alpha_1(\partial_x u^+)^2+\alpha_2(\partial_y u^+)^2+\alpha_1(\partial_x u^-)^2+\alpha_2(\partial_y u^-)^2\,\D x\,\D y,\label{elas formula}
\end{align}
where notations $\partial_x u=\frac{\partial u}{\partial x}$, $\partial_y u=\frac{\partial u}{\partial y}$, and parameters
\begin{align}
    \alpha_1=&\frac{1}{3\sqrt{3}a^2}\sum_{\langle\vs_1,\vs_2\rangle}\sum_{i+j=2}\tbinom{2}{i}\Big\{\left[\partial_{ij}V\left(\frac{\vs_1}{a},\frac{\vs_2}{a}\right)\right]_{11}\left(\frac{\vs_1}{a}\cdot \ve_x\right)^i\left(\frac{\vs_2}{a}\cdot\ve_x\right)^j\notag\\&+\frac{3}{2}\left[\partial_{ij}V\left(\frac{\vs_1\pm \vp}{a},\frac{\vs_2}{a}\right)\right]_{11}\left[\left(\frac{\vs_1\pm \vp}{a}\right)\cdot\ve_x\right]^i\left[\left(\frac{\vs_2}{a}\right)\cdot\ve_x\right]^j\Big\},\label{eqn:alpha1}\end{align}\begin{align} \alpha_2=&\frac{1}{3\sqrt{3}a^2}\sum_{\langle\vs_1,\vs_2\rangle}\sum_{i+j=2}\tbinom{2}{i}\Big\{\left[\partial_{ij}V\left(\frac{\vs_1}{a},\frac{\vs_2}{a}\right)\right]_{11}\left(\frac{\vs_1}{a}\cdot \ve_y\right)^i\left(\frac{\vs_2}{a}\cdot\ve_y\right)^j\notag\\&+\frac{3}{2}\left[\partial_{ij}V\left(\frac{\vs_1\pm \vp}{a},\frac{\vs_2}{a}\right)\right]_{11}\left[\left(\frac{\vs_1\pm \vp}{a}\right)\cdot\ve_y\right]^i\left[\left(\frac{\vs_2}{a}\right)\cdot\ve_y\right]^j\Big\},\label{eqn:alpha2}\\ \ve_x=&(1,0),\ \ve_y=(0,1).\notag
\end{align}
Here $[\mA]_{11}$ is the $(1,1)$-entry of the matrix $\mA\in \mathbb{R}^{2\times2}$, and also for  simplicity of notations, for  $V(\pmb \rho_1,\pmb \rho_2)$ with $\pmb \rho_1=(\rho_{11},\rho_{12})$ and $\pmb \rho_2=(\rho_{21},\rho_{22})$, we denote
\begin{align}
&{\textstyle \partial_{1}V=(\frac{\partial V}{\partial{\rho_{11}}},\frac{\partial V}{\partial{\rho_{12}}})^\T,~\partial_{2}V=(\frac{\partial V}{\partial{\rho_{21}}},\frac{\partial V}{\partial{\rho_{22}}})^\T},\notag\\&\partial_{20}V=\begin{bmatrix}
\frac{\partial^2 V}{\partial{\rho_{11}}^2}& \frac{\partial^2 V}{\partial \rho_{11}\partial \rho_{12}}\vspace{1ex}\\
\frac{\partial^2 V}{\partial{\rho_{12}}\partial \rho_{11}}& \frac{\partial^2 V}{\partial{\rho_{12}}^2}
\end{bmatrix},~\partial_{11}V=\begin{bmatrix}
\frac{\partial^2 V}{\partial{\rho_{11}}\partial \rho_{21}} & \frac{\partial^2 V}{\partial{\rho_{11}}\partial \rho_{22}} \vspace{1ex}\\ \frac{\partial^2 V}{\partial{\rho_{12}}\partial \rho_{21}} & \frac{\partial^2 V}{\partial{\rho_{12}}\partial \rho_{22}}
\end{bmatrix}, \notag\\&\partial_{02}V=\begin{bmatrix}
\frac{\partial^2 V}{\partial{\rho_{21}}^2}&\frac{\partial^2 V}{\partial{\rho_{21}}\partial \rho_{22}} \vspace{1ex}\\ \frac{\partial^2 V}{\partial{\rho_{22}}\partial \rho_{21}}  & \frac{\partial^2 V}{\partial{\rho_{22}}^2}
\end{bmatrix}.\notag
\end{align}

In fact, by using Taylor expansion in the elastic energy in the atomistic model $E_{\text{a-el}}^+$ in Eq.~\eqref{eqn:elastic-atomistic},
 we can get the following leading order approximation of the continuum elastic energy in the upper layer:
\begin{align}
&E^+_{\text{elas}}=\frac{2}{\sqrt{3}a^4}\int_{\mathbb{R}^2}\sum_{\langle\vs_1,\vs_2\rangle}\sum_{i+j=2}\tbinom{2}{i}\Big\{\frac{1}{6}\partial_{ij}V\left(\frac{\vs_1}{a},\frac{\vs_2}{a}\right)
\big[(\vs_1\cdot\nabla)\vu^+\big]^i\big[(\vs_2\cdot\nabla)\vu^+\big]^j\notag\\&+\frac{1}{4}\partial_{ij}V\left(\frac{\vs_1\pm \vp}{a},\frac{\vs_2}{a}\right)\big[(\vs_1\pm \vp)\cdot(\nabla \vu^+)\big]^i\big[(\vs_2)\cdot(\nabla \vu^+)\big]^j\Big\}\,\D x\,\D y.\label{elas no}\end{align}
This leading order approximation is easily obtained by using the approximations
 \begin{align}\frac{\vu^{A+}_{\vx+\vs_1}-\vu^{A+}_{\vx}}{a}&\approx \left(\frac{\vs_1}{a}\cdot\nabla\right)\vu^+,\ \frac{\vu^{B+}_{\vx+\vs_1}-\vu^{B+}_{\vx}}{a}\approx \left(\frac{\vs_1}{a}\cdot\nabla\right)\vu^+\notag\\\frac{\vu^{A+}_{\vx+\vs_1}-\vu^{B+}_{\vx}}{a}&\approx \left(\frac{\vs_1-\vp}{a}\cdot\nabla\right)\vu^+,\ \frac{\vu^{B+}_{\vx+\vs_1}-\vu^{A+}_{\vx}}{a}\approx \left(\frac{\vs_1+\vp}{a}\cdot\nabla\right)\vu^+,\notag\end{align}
and symmetry of the lattice.
For simplicity of notations, in Eqs.~\eqref{elas no}, we have used the expression
 \begin{align}
\partial_{ij} V(\pmb \rho_1,\pmb \rho_2)\vxi_1\vxi_2:=\vxi_1 \partial_{ij} V(\pmb \rho_1,\pmb \rho_2) \vxi_2^\T, \ \text{ for }\vxi_i\in\mathbb{R}^2.
\notag
\end{align}

The formula of the elastic energy in the lower layer $E^-_\text{elas}$ in the form of Eq.~\eqref{elas no} is the same except that $\vs_1,\vs_2\in \mathbb{L}^* $ with $\mathbb{L}^* $ being defined in Eq.~\eqref{condition11} instead of Eq.~\eqref{condition}, and all the variables with notation ``$+$" are replaced by those with notation ``$-$".

The two forms of elastic energy densities in Eqs.~\eqref{elas formula} and \eqref{elas no} are identical. Calculation from Eq.~\eqref{elas no} to Eq.~\eqref{elas formula} will be summarized as a lemma below. (Same for the calculation for the corresponding two forms of $E^-_\text{elas}$.) In the PN model part (Theorem \ref{Theorem 1}), we use formula \eqref{elas formula}. In the atomistic model part (consistency of PN model and stability of atomistic model), we use formula \eqref{elas no}.

\newtheorem{mylem}{Lemma}
\begin{mylem}\label{no xy}
The two forms of elastic energy densities in Eqs.~\eqref{elas formula} and \eqref{elas no} are identical.
 \end{mylem}
 \begin{proof}

 We derive Eq.~\eqref{elas formula} from Eq.~\eqref{elas no} by direct calculation.  It is easy to see that the coefficients of $(\partial_x u)^2$ and $(\partial_y u)^2$  in Eq.~\eqref{elas formula} are the same as those in Eq.~\eqref{elas no}. It remains to show that there is no $\partial_x u \partial_y u$ term in Eq.~\eqref{elas no}, which will be done using lattice symmetry as follows.

Recall that Eq.~\eqref{elas no} is obtained  by Taylor expansions in the elastic energy in the atomistic model in Eq.~\eqref{eqn:elastic-atomistic}. In the atomistic model, we divide the elastic interactions within this layer into four types of three-body interactions: $AAA$, $BBB$, $AAB$, and $ABB$. Recall that here we use three-body potential.
 The three atoms in a three-body interaction form a triangle; see Fig.~\ref{aaa-aab}. Without loss of generality, suppose that the point being considered is an $A$-atom, denoted by $T$. In the elastic energy in the atomistic model in Eq.~\eqref{eqn:elastic-atomistic}, when  $\vx=T$ in the outer summation, the three-body interactions in the summation includes only those of $AAA$ and $AAB$ types that contain the atom $T$; see Fig.~\ref{aaa-aab}. The interactions among such a triangle is expressed in terms of the two edge vectors $\vs_1$ and $\vs_2$ starting from the atom $T$ as given in the inner summations in Eq.~\eqref{eqn:elastic-atomistic}. Recall that the interactions of $ABB$ and $BBB$ types are included in summations with a $B$-atom in the outer summation in Eq.~\eqref{eqn:elastic-atomistic}.

\begin{figure}[htbp]
\centering
\includegraphics[width=\textwidth]{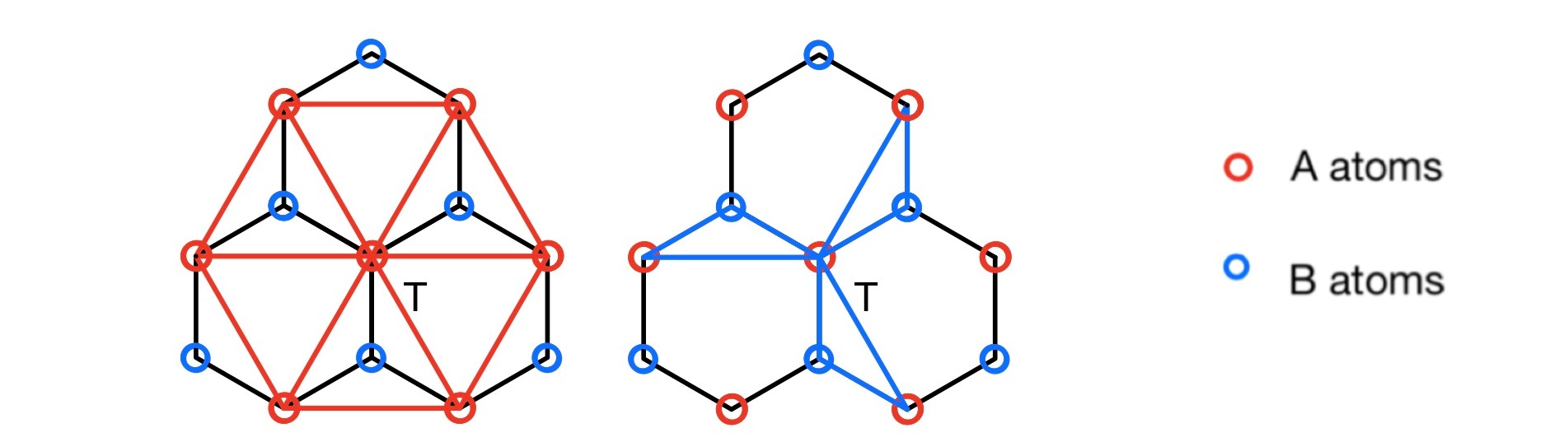}
	\caption{Different types of three-body interactions in a layer. Left: AAA, right: AAB. Types BBB and ABB are similar.}
	\label{aaa-aab} 
\end{figure}

The complex lattice being considered has $3$-fold rotational symmetry, i.e., it does not change after rations of $2\pi/3$ and $4\pi/3$.
For an AA edge vector starting from the atom $T$, e.g.,  $\vs_1=b_1\ve_1+b_2\ve_2=a\left(b_1+\frac{1}{2}b_2,\frac{\sqrt{3}}{2}b_2\right)$, where $b_1,b_2\in \mathbb{Z}$, it becomes $b_2\ve_1-(b_1+b_2)\ve_2$ and $-(b_1+b_2)\ve_1+b_1\ve_2$ after rotations of $2\pi/3$ and $4\pi/3$ around atom $T$, respectively.

Consider an $AAA$ interaction associated with the atom $T$.
Assume that the two AA edge vectors starting from atom $T$ are $$\vs_1= b_1\ve_1+b_2\ve_2,~~\vs_2= b_3\ve_1+b_4\ve_2.$$
The three-body interaction among this $AAA$ triangle is included in the summation in the first line in Eq.~\eqref{eqn:elastic-atomistic}. Due to the $3$-fold rotational symmetry, interactions among AAA triangles obtained by $2\pi/3$ and $4\pi/3$ rotations from this one are also included in this summation.

Recall that the continuum model in Eq.~\eqref{elas no} is obtained by leading order approximation from the atomistic model in Eq.~\eqref{eqn:elastic-atomistic}.  We divide the first term in Eq.~\eqref{elas no} for the $AAA$ interactions into three parts that contain the terms
$\partial_{20} V\left(\frac{\vs_1}{a},\frac{\vs_2}{a}\right),\partial_{11} V\left(\frac{\vs_1}{a},\frac{\vs_2}{a}\right)$ and $\partial_{02} V\left(\frac{\vs_1}{a},\frac{\vs_2}{a}\right)$, respectively.
For the $\partial_{20} V\left(\frac{\vs_1}{a},\frac{\vs_2}{a}\right)$ term, due to the  $3$-fold rotational symmetry, its coefficient can be written as
 \begin{align}
     &\lambda_1\Bigg\{\left[\left(b_1+\frac{1}{2}b_2\right)\partial_x u+\frac{\sqrt{3}}{2}b_2\partial_y u\right]^2+\left[\left(b_2+\frac{1}{2}b_1\right)\partial_x u-\frac{\sqrt{3}}{2}b_1\partial_y u\right]^2\notag\\&+\left[\left(\frac{1}{2}b_2-\frac{1}{2}b_1\right)\partial_x u-\frac{\sqrt{3}}{2}\left(b_1+b_2\right)\partial_y u\right]^2\Bigg\}\notag\\=&\lambda_{11}\left(\partial_x u\right)^2+\lambda_{12}\left(\partial_y u\right)^2,\notag
 \end{align}
 where  $\lambda_1,~\lambda_{11},~\text{and }\lambda_{12}$ are independent
  of $\partial_x u$ and $\partial_y u$.
 For the term $\partial_{11} V\left(\frac{\vs_1}{a},\frac{\vs_2}{a}\right)$, for the same reason, its coefficient can be written as
 \begin{align}
    &\lambda_2\Bigg\{\left[\left(b_1+\frac{1}{2}b_2\right)\partial_x u+\frac{\sqrt{3}}{2}b_2\partial_y u\right] \left[\left(b_4+\frac{1}{2}b_3\right)\partial_x u-\frac{\sqrt{3}}{2}b_3\partial_y u\right]\notag\\&+\left[\left(b_2+\frac{1}{2}b_1\right)\partial_x u-\frac{\sqrt{3}}{2}b_1\partial_y u\right]\left[\left(\frac{1}{2}b_4-\frac{1}{2}b_3\right)\partial_x u-\frac{\sqrt{3}}{2}(b_3+b_4)\partial_y u\right]\notag\\&-\left[\left(\frac{1}{2}b_2-\frac{1}{2}b_1\right)\partial_x u-\frac{\sqrt{3}}{2}(b_1+b_2)\partial_y u\right]\left[\left(b_3+\frac{1}{2}b_4\right)\partial_x u+\frac{\sqrt{3}}{2}b_4\partial_y u\right]\Bigg\}\notag\\=&\lambda_{21}\left(\partial_x u\right)^2+\lambda_{22}\left(\partial_y u\right)^2, \notag
 \end{align}
 where $\lambda_2,~\lambda_{21},~\text{and }\lambda_{22}$ are independent
  of $\partial_x u$ and $\partial_y u$.
   The coefficient of    $V_{02}\left(\frac{\vs_1}{a},\frac{\vs_2}{a}\right)$ is the same as that of  $\partial_{20} V\left(\frac{\vs_1}{a},\frac{\vs_2}{a}\right)$ given above except that $b_1$ and $b_2$ are replaced by $b_3$ and $b_4$.

    Thus we know that the $u_xu_y$ term vanishes in Eq.~\eqref{elas no} for all the $AAA$ interactions. We can also show similarly that  the $\partial_x u \partial_y u$ term vanishes in Eq.~\eqref{elas no} for all the $AAB$ interactions. A slight difference in the calculation is that for the $AB$ edge vector starting from the atom, it is $c_1\ve_1+c_2\ve_2+\vp$, where $c_1,c_2\in \mathbb{Z}$, and it becomes  $c_2\ve_1-(c_1+c_2+1)\ve_2+\vp$ and $-(c_1+c_2+1)\ve_1+c_1\ve_2+\vp $ after rotations of $2\pi/3$ and $4\pi/3$ around atom $T$.

\end{proof}

\section{Rescaling}

Due to the fact that the inter-layer van der Waals-like interaction is much weaker than the intra-layer  covalent-bond interaction in a bilayer graphene~\cite{dai2016structure}, i.e., ${V_{\rm d}} \ll {V}$ in the atomistic model, we assume that
$$V_{\rm d}/V=O(\varepsilon ^2),$$
where $\varepsilon$ is a dimensionless small parameter.
 In the PN model of bilayer graphene system, the intra-layer  covalent-bond interaction gives the elastic energy $E_\text{elas}$, and the inter-layer  van der Waals-like interaction gives the misfit energy $E_\text{mis}$. Following the convergence analysis for simple lattices in Ref.~\cite{luo2018atomistic}, we define the small parameter $\varepsilon$ based on the PN model as
\begin{equation}
\varepsilon=\sqrt{\frac{a^2\gamma_{xx}(\vzero)}{\sqrt{\alpha_1\alpha_2}}},
\end{equation}
and assume that
\begin{equation}
\varepsilon\ll 1.
\end{equation}

Using $a/\varepsilon$ as the  length unit for the spatial variable $x$ and $a$ as the length unit for the displacement,  we define the rescaled quantities:$$\bar{\vx}=\frac{\varepsilon \vx}{a},\ \bar{\vu}^{ \pm}=\frac{\vu^{ \pm}}{a},\ \bar{\vphi}=\frac{\vphi}{a}, \ \bar{s}_i=\frac{s_i}{a}, \ \bar{p}=\frac{p}{a},
\ \bar{d}=\frac{d}{a},$$
$$\bar{\alpha_i}=a^2\alpha_i,\qquad \bar{\gamma}(\bar{\vphi})=a^2\gamma(\vphi).$$

The elastic energy in the PN model in Eq.~\eqref{elas no} becomes
\begin{align}
E^+_{\text{elas}}=&\frac{2}{\sqrt{3}}\int_{\mathbb{R}^2}\sum_{\langle\bar{\vs}_1,\bar{\vs}_2\rangle}\sum_{i+j=2}
\tbinom{2}{i}\Big\{\frac{1}{6}\partial_{ij}V(\bar{\vs}_1,\bar{\vs}_2)
\big[(\bar{\vs}_1\cdot\nabla)\bar{\vu}^+\big]^i\big[(\bar{\vs}_2\cdot\nabla)\bar{\vu}^+\big]^j\notag\\
&+\frac{1}{4}\partial_{ij}V(\bar{\vs}_1\pm \bar{\vp},\bar{\vs}_2)\big[(\bar{\vs}_1\pm \bar{\vp})\cdot(\nabla \bar{\vu}^+)\big]^i\big[(\bar{\vs}_2)\cdot(\nabla \bar{\vu}^+)\big]^j\Big\}\,\D \bar{x}\,\D \bar{y},\label{elas rescale}\end{align}
and similar formula for $E^-_{\text{elas}}$.
The misfit energy in Eq.~\eqref{eqn:misfit00} in the PN model becomes
\begin{equation}
E_{\text{mis}} = \int_{\mathbb{R}^2} \bar{\gamma}(\bar{\vphi})\,\D \bar{x}\,\D \bar{y}\label{mis rescale},
\end{equation}
where
\begin{flalign}
\bar{\gamma}(\bar{\vphi})\notag=&\frac{2}{\sqrt{3}}
\sum_{\vs}\Big[U(\bar{\vphi}+\bar{\vs}-\bar{\vd})-U(\bar{\vs}-\bar{\vd})\notag\\
&+U(\bar{\vphi}+\bar{\vs}-\bar{\vd}+\bar{\vp})
-U(\bar{\vs}-\bar{\vd}+\bar{\vp})\notag\\
&+U(\bar{\vphi}+\bar{\vs}-\bar{\vd}-\bar{\vp})-U(\bar{\vs}-\bar{\vd}-\bar{\vp})\notag\\
&+U(\bar{\vphi}+\bar{\vs}-\bar{\vd})-U(\bar{\vs}-\bar{\vd})\Big].
\end{flalign}
Here, we have used the notation
\begin{equation}
U=\varepsilon^{-2}V_{\rm d}.
\end{equation}

Using the above notations and the intra-layer interaction energy in Eq.~\eqref{eqn:elastic-atomistic}
and the inter-layer interaction energy in Eq.~\eqref{eqn:a-mis} in the atomistic model, the total energy in the atomistic model  is
\begin{align}
E_{\text{a}}[\bar{\vu}]=\sum_{\bar{\vx}}e_{\bar{\vx}}[\bar{\vu}],
\end{align}
where
\begin{align}
e_{\bar{\vx}}[\bar{\vu}]
:=&\frac{1}{6}\sum_{\langle\bar{\vs}_1,\bar{\vs}_2\rangle}
\big[V(\bar{\vs}_1+{\bar{\vu}^{A\pm}_{\bar{\vx}+\bar{\vs}_1}-\bar{\vu}^{A\pm}_{\bar{\vx}}},\bar{\vs}_2
+{\bar{\vu}^{A\pm}_{\bar{\vx}+\bar{\vs}_2}
-\bar{\vu}^{A\pm}_{\bar{\vx}}})\notag\\
&+V(\bar{\vs}_1+{\bar{\vu}^{B\pm}_{\bar{\vx}+\bar{\vs}_1}-\bar{\vu}^{B\pm}_{\bar{\vx}}},\bar{\vs}_2
+{\bar{\vu}^{B\pm}_{\bar{\vx}+\bar{\vs}_2}-\bar{\vu}^{B\pm}_{\bar{\vx}}})-2V_3(\bar{\vs}_1,\bar{\vs}_2)\big]\notag\\
&+\frac{1}{2}\sum_{\langle\bar{\vs}_1,\bar{\vs}_2\rangle}\big[V(\bar{\vp}+\bar{\vs}_1
+{\bar{\vu}^{B\pm}_{\bar{\vx}+\bar{\vs}_1}
-\bar{\vu}^{A\pm}_{\bar{\vx}}},\bar{\vs}_2+{\bar{\vu}^{A\pm}_{\bar{\vx}+\bar{\vs}_2}-\bar{\vu}^{A\pm}_{\bar{\vx}}})\notag\\
&+V(-\bar{\vp}+\bar{\vs}_1+{\bar{\vu}^{A\pm}_{\bar{\vx}+\bar{\vs}_1}
-\bar{\vu}^{B\pm}_{\bar{\vx}}},\bar{\vs}_2+{\bar{\vu}^{B\pm}_{\bar{\vx}+\bar{\vs}_2}
-\bar{\vu}^{B\pm}_{\bar{\vx}}})\notag\\
&-V(-\bar{\vp}+\bar{\vs}_1,\bar{\vs}_2)-V(\bar{\vp}+\bar{\vs}_1,\bar{\vs}_2)\big]\notag\\
&+\varepsilon^2\sum_{\bar{\vs}}\big[U({\bar{\vu}^{A+}_{\bar{\vx}+\bar{\vs}}-\bar{\vu}^{A-}_{\bar{\vx}}}
+\bar{\vs}-\bar{\vd})-U(\bar{\vs}-\bar{\vd})\notag\\
&+U({\bar{\vu}^{A+}_{\bar{\vx}+\bar{\vs}}-\bar{\vu}^{B-}_{\bar{\vs}_1}}+\bar{\vs}-\bar{\vd}+\bar{\vp})
-U(\bar{\vs}-\bar{\vd}+\bar{\vp})\notag\\
&+U({\bar{\vu}^{B+}_{\bar{\vx}+\bar{\vs}}-\bar{\vu}^{A-}_{\bar{\vx}}}+\bar{\vs}-\bar{\vd}-\bar{\vp})
-U(\bar{\vs}-\bar{\vd}-\bar{\vp})\notag\\
&+U({\bar{\vu}^{B+}_{\bar{\vx}+\bar{\vs}}-\bar{\vu}^{B-}_{\bar{\vx}}}+\bar{\vs}-\bar{\vd})-U(\bar{\vs}-\bar{\vd})\big].
\label{at re}
\end{align}


For simplicity, we will still use letters without bar for the rescaled variables in later sections.

\section{Total Energy Revisited and Assumptions}

\subsection{Total Energy Revisited}

Although we are interested in a straight dislocation in the continuum model, the atomic structure is not uniform along the dislocation using the atomistic model due to the complex lattice. Therefore, we need to consider two-dimensional models for the slip plane of the dislocation instead of one-dimensional models in the convergence from the atomistic model to the continuum, PN model. We will show that the dislocation in equilibrium is a straight line in the PN model (in Theorem \ref{Theorem 1}).
Here we define the total energy in the PN model as
\begin{equation}
    E_\text{PN}[\vu]=\lim_{R\to+\infty}\frac{1}{2R}\int_{-R}^{R} \D y\int_{\mathbb{R}}W(\nabla u^+)+W(\nabla u^-)+\gamma(\vphi)\,\D x\,\label{re PN}
\end{equation}
where  $W(\mA)$ is defined in \eqref{elas formula}.

Accordingly,  we redefine the total energy in atomistic model as
\begin{align}
E_{\text{a}}[\vu]:=&\lim_{R\to+\infty}\frac{1}{2R}\sum_{\vx\in\mathbb{L}_{R}}e_{\vx}[\vu],
\label{at red}
\end{align}
where $e_{\vx}[\vu]$ is defined in \eqref{at re} (with bars being omitted for simplicity of notations as remarked above), and $\mathbb{L}_{R}$ is a truncation of $\mathbb{L}$:
\begin{equation}
  \mathbb{L}_{R}:=\big\{\vs=(x,y)\in\mathbb{L}\mid|y|\le R\big\}.\label{reset}
\end{equation}

\subsection{Notations}
We collect some notations for the rest of the paper here.
First, as discussed in the previous two sections, we assume that the displacement vector is always in the $x$ direction due to the straight edge dislocation in both the atomistic model and the PN model, and focus on the $x$-component of the displacement. Sometimes, when the interaction potential $V(\cdot,\cdot)$ defined in Eq.~\eqref{eqn:3bd} is used, it is more convenient to use the vector displacement, and in this case, we will use the vector $\mathbf u=(u,0)$ instead of the scalar displacement component $u$, in both the atomistic model and the PN model.

	For convenience, we denote  $\mathbb{L}_{\vp}=\mathbb{L}\cup(\mathbb{L}+\{\vp\})$, where $\mathbb{L}$ is defined in \eqref{condition}. Then we define the difference operators $D_{\vs}$ and $D_{\vs}^{\vp}$ for $f$ on $\mathbb{L}_{\vp}$: \begin{equation}
	D_{\vs} f(\vx)= \frac{f(\vx+\varepsilon\vs)-f(\vx)}{\varepsilon},~ D^{\pm\vp}_{\vs} f(\vx)= \frac{f(\vx+\varepsilon\vs\pm\varepsilon\vp)-f(\vx)}{\varepsilon}.\label{D}
	\end{equation}
	where $\vs\in \sL$ defined in (\ref{condition}) and $\vp=\left(\frac{1}{2},\frac{3}{6}\right)$. Moreover, for functions defined on $\sL_p$, we denote \begin{equation}
	f_{\vs}^A:=f(\varepsilon\vs),~f_{\vs}^B:=f(\varepsilon\vs+\varepsilon\vp).
	\end{equation}
	
	Now we introduce discrete Sobolev space $$H_\varepsilon^1=H_\varepsilon^1(\sL_{\vp})=\{f\mid \|f\|_{1,\varepsilon}<\infty\}$$ with \begin{align}\|f\|_{1,\varepsilon}^2=&\lim_{R\to+\infty}\frac{\varepsilon^2}{2R}\sum_{\vx\in\mathbb{L}_{R}}
\left(\sum_{\|\vs\|\le 1}\left|D_{\vs}f^A_{\vx}\right|^2+\left|D^{\vp}_{\vs}f^A_{\vx}\right|^2+\left|D_{\vs}f^B_{\vx}\right|^2
+\left|D_{\vs}^{-\vp}f^B_{\vx}\right|^2\right)\notag\\
&+\lim_{R\to+\infty}\frac{\varepsilon^2}{2R}\sum_{\vx\in\mathbb{L}_{R}}\left(\sum_{\|\vs\|\le 1}\left|f_{\vs}^B\right|^2+\left|f_{\vs}^A\right|^2\right)\notag.\end{align}
	Similarly, we define the continuum Sobolev space on $\sR^2$ $$H_e^1=H_e^1(\sR^2)=\{f\mid\|f\|_{1}\le\infty\}$$ with $$\|f\|_1^2=\lim_{R\to+\infty}\frac{1}{2R}\int_{-R}^{R}\int_{\mathbb{R}}
|\partial_x f|^2+|\partial_y f|^2+|f|^2\,\D x\,\D y.$$
	
	For simplicity of notations, we write $f=\{f^+,f^-\}\in H_\varepsilon^1$ if $f^\pm\in H_\varepsilon^1$ and  $f=\{f^+,f^-\}\in H^1_e$ if $f^\pm\in H^1_e$. If $f=\{f^+,f^-\}$, we denote
	\begin{align}
		f^\perp=f^+-f^-.
	\end{align}

	Now we define the test function space $X_0\subset H^1_e$ for the PN model on $\{\sR^2,\sR^2\}$ and the test function space $X_\varepsilon\subset H^1_\varepsilon$ for atomistic model on $\{\sL_{\vp},\sL_{\vp}\}$:
	\begin{align}
		X_0=\{f\in H^1_e\mid\|f\|_{X_0}<+\infty,~f^\pm(0,y)=0,~\lim_{x\to\pm\infty}f^\pm(x,y)=0\},
	\end{align}
	where $\|f\|_{X_0}^2=(f,f)_{X_0}$, and	
	\begin{align}
(f,g)_{X_0}=&\alpha_1\big[(\partial_x f^+,\partial_x g^+)_v+(\partial_x f^-,\partial_x g^-)_v\big]\notag\\
&+\alpha_2\big[(\partial_y f^+,\partial_y g^+)+(\partial_y f^-,\partial_y g^-)_v\big]
+\frac{4\sqrt{3}}{3}( f^\perp,\vg^\perp)_v, \label{norm} \\
(f,g)_v=&\lim_{R\to+\infty}\frac{1}{2R}\int_{-R}^{R}\int_{\mathbb{R}}fg\,\D x\,\D y. \notag
	\end{align}
	\begin{align}
	X_\varepsilon=\Bigg\{&f\in H^1_\varepsilon\mid\|f\|_{X_\varepsilon}<\infty, \lim_{\vs\cdot\ve_1\to\infty}f_{\vs}^{\kappa\pm}=0,~f_{\vs}^{\kappa\pm}=-f_{-\vs}^{\kappa\pm},\notag
\\&\left\{f^+,f^-\right\}+\left\{\frac{1}{4},-\frac{1}{4}\right\}\text{ satisfies ADC. }\Bigg\}
	\end{align}
	where $\ve_1=(1,0)$, $\kappa=A,~B$, $\|f\|_{X_\varepsilon}^2=(f,f)_{X_\varepsilon}$, and definition of the Atomistic Dislocation Condition (ADC) will be given in Section~\ref{sec:adc}. We denote $(\cdot,\cdot)_{X_\varepsilon}$ as
	\begin{align} (f,g)_{X_\varepsilon}&=(f^+,g^+)_a+(f^-,g^-)_a+(f^\perp,g^\perp)_\varepsilon\\
(f,g)_\varepsilon&=\lim_{R\to+\infty}\frac{1}{2R}\varepsilon^2\sum_{\vx\in\mathbb{L}_{R}} f_{\vx}^Ag_{\vx}^A+f_{\vx}^Bg_{\vx}^B\notag
\\(f,g)_a&=(f,g)_A+(f,g)_B,
\end{align}
where $(f,g)_A$ and $(f,g)_B$ are defined by
	\begin{align}
	(f,g)_A
=&\lim_{R\to+\infty}\frac{1}{2R}\sum_{\vx\in\mathbb{L}_{R}}
\sum_{\pm}\Bigg\{\frac{1}{6}\varepsilon^2
\sum_{\langle\vs_1,\vs_2\rangle}\partial_{20}V(\vs_1,\vs_2)(D_{\vs_1}f_{\vx}^{A\pm})
(D_{s_1}g_{\vx}^{A\pm})\notag\\
&+\frac{1}{6}\varepsilon^2\sum_{\langle\vs_1,\vs_2\rangle}
\partial_{11}V(\vs_1,\vs_2)\big[(D_{\vs_1}f_{\vx}^{A\pm})(D_{s_2}g_{\vx}^{A\pm})
+(D_{\vs_2}g_{\vx}^{A\pm})(D_{s_1}f_{\vx}^{A\pm})\big]\notag\\
&+\frac{1}{6}\varepsilon^2\sum_{\langle\vs_1,\vs_2\rangle}
\partial_{02}V(\vs_1,\vs_2)(D_{\vs_2}f_{\vx}^{A\pm})(D_{s_2}g_{\vx}^{A\pm})\notag\\
&+\varepsilon^2\frac{1}{2}\sum_{\langle\vs_1,\vs_2\rangle}
\partial_{20}V(\vp+\vs_1,\vs_2)(D^{\vp}_{\vs_1}f_{\vx}^{A\pm})(D^{\vp}_{\vs_1}g_{\vx}^{A\pm})\notag\\
&+\varepsilon^2\frac{1}{2}\sum_{\langle\vs_1,\vs_2\rangle}
\partial_{11}V(\vp+\vs_1,\vs_2)\big[(D^{\vp}_{\vs_1}f_{\vx}^{A\pm})(D_{\vs_2}g_{\vx}^{A\pm})
+(D^{\vp}_{\vs_1}g_{\vx}^{A\pm})(D_{\vs_2}f_{\vx}^{A\pm})\big]\notag\\
&+\varepsilon^2\frac{1}{2}\sum_{\langle\vs_1,\vs_2\rangle}
\partial_{02}V(\vp+\vs_1,\vs_2)(D^{\vp}_{\vs_2}f_{\vx}^{A\pm})(D_{\vs_2}g_{\vx}^{A\pm})\Bigg\},
	\end{align}
	and similar definition is used for $(f,g)_B$ by changing $\vp$ into $-\vp$.
	
	It is easy to check that $X_0$ and $X_\varepsilon$ are Hirbert spaces with  inner products $(\cdot,\cdot)_{X_0}$ and $(\cdot,\cdot)_{X_\varepsilon}$, respectively. We use the notations $\langle\cdot,\cdot\rangle_{0}$ and $\langle\cdot,\cdot\rangle_{\varepsilon}$ for pairing on $X^*_0\times X_0$ and $X^*_\varepsilon\times X_\varepsilon$, respectively.

	Finally, we define the solution spaces for the PN model and the atomistic model for an edge dislocation:
	\begin{align}
	S_0&=\left\{f\mid f-u_0\in X_0\right\},\notag\\
	S_\varepsilon&=\left\{f\mid f-u_\varepsilon\in X_\varepsilon\right\},\label{s0}
	\end{align}
where $u_0=\{u_0^+,u_0^-\}$ and $u_\varepsilon=\{u_\varepsilon^+,u_\varepsilon^-\}$ \begin{align}
u_0^+(\vx)&=u_0^+(x,y)=\frac{1}{2\pi}\arctan x+\frac{1}{4},\notag\\
u_0^-(\vx)&=u_0^-(x,y)=-\frac{1}{2\pi}\arctan x-\frac{1}{4},\notag\\
(u_\varepsilon^\pm)^A_{\vs}&=u_0^\pm(\varepsilon\vs),
~(u_\varepsilon^\pm)^B_{\vs}=u_0^\pm(\varepsilon\vs+\varepsilon\vp).
\end{align}
These definitions are adopted to accommodate the boundary conditions in Eq.~\eqref{eqn:phi1}--\eqref{eqn:phi3} and Eq.~\eqref{eqn:aphi} for the edge dislocation in the PN model and the atomistic model.

In the proof of stability in Section~\ref{sec:stability}, we divide the energy in each of the atomistic model and the PN model into two parts:
\begin{align}
&E_a^A[\vu]\notag\\=&\lim_{R\to+\infty}\frac{1}{2R}\sum_{\vx\in\mathbb{L}_{R}}\Big\{\frac{1}{6}\sum_{\langle\vs_1,\vs_2\rangle}\big[V(\vs_1+{\vu^{A\pm}_{\vx+\vs_1}-\vu^{A\pm}_{\vx}},\vs_2+{\vu^{A\pm}_{\vx+\vs_2}-\vu^{A\pm}_{\vx}})\big]\notag\\&+\frac{1}{2}\sum_{\langle\vs_1,\vs_2\rangle}\big[V(\vp+\vs_1+{\vu^{B\pm}_{\vx+\vs_1}-\vu^{A\pm}_{\vx}},\vs_2+{\vu^{A\pm}_{\vx+\vs_2}-\vu^{A\pm}_{\vx}})\notag\\&-\frac{1}{3}V(\vs_1,\vs_2)-V(\vp+\vs_1,\vs_2)\big]+\varepsilon^2\sum_{\vs}\big[U({\vu^{A+}_{\vx+\vs}-\vu^{A-}_{\vx}}+\vs-\vd)\notag\\&+U({\vu^{B+}_{\vx+\vs}-\vu^{A-}_{\vx}}+\vs-{\vp}-\vd)-U(\vs-\vd)-U(\vs-{\vp}-\vd)\big]\Big\}\label{A A part},\end{align}
\begin{align}
&E_a^B[\vu]\notag\\=&\lim_{R\to+\infty}\frac{1}{2R}\sum_{\vx\in\mathbb{L}_{R}}\Big\{\frac{1}{6}\sum_{\langle\vs_1,\vs_2\rangle}\big[V(\vs_1+{\vu^{B\pm}_{\vx+\vs_1}-\vu^{B\pm}_{\vx}},\vs_2+{\vu^{B\pm}_{\vx+\vs_2}-\vu^{B\pm}_{\vx}})\notag\\&+\frac{1}{2}\sum_{\langle\vs_1,\vs_2\rangle}\big[V(-\vp+\vs_1+{\vu^{A\pm}_{\vx+\vs_1}-\vu^{B\pm}_{\vx}},\vs_2+{\vu^{B\pm}_{\vx+\vs_2}-\vu^{B\pm}_{\vx}})\notag\\&-\frac{1}{3}V(\vs_1,\vs_2)-V(-{\vp}+\vs_1,\vs_2)\big]+\varepsilon^2\sum_{\vs}\big[U({\vu^{A+}_{\vx+\vs}-\vu^{B-}_{\vx}}+\vs-\vp-\vd)\notag\\&-U(\vs-\vp-\vd)+U({\vu^{B+}_{\vx+\vs}-\vu^{B-}_{\vx}}+\vs-\vd)-U(\vs-\vd)\big]\Big\}\label{A B part},\end{align}
\begin{align}
&E_\text{PN}^A[\vu]\notag\\=&\lim_{R\to+\infty}\frac{1}{\sqrt{3}R}\int_{-R}^{R}\int_{\mathbb{R}}\sum_{\langle\vs_1,\vs_2\rangle}\sum_{i+j=2}\tbinom{2}{i}\Big\{\frac{1}{12}\partial_{ij}V(\vs_1,\vs_2)
\big[(\vs_1\cdot\nabla){\vu}^\pm\big]^i\notag\\&\cdot\big[(\vs_2\cdot\nabla){\vu}^\pm\big]^j+\frac{1}{4}\partial_{ij}V(\vs_1+ \vp,\vs_2)\big[(\vs_1+\vp)\cdot(\nabla {\vu}^\pm)\big]^i\big[(\vs_2)\cdot(\nabla {\vu}^\pm)\big]^j\Big\}\notag\\&+\sum_{\vs}\big[U(\vphi+\vs-\vd)+U(\vphi+\vs+{\vp}-\vd)-U(\vs-\vd)-U(\vs+{\vp}-\vd)\big]\,\D x\,\D y,\label{PN a}
\end{align}
\begin{align}
&E_\text{PN}^B[\vu]\notag\\=&\lim_{R\to+\infty}\frac{1}{\sqrt{3}R}\int_{-R}^{R}\int_{\mathbb{R}}\sum_{\langle\vs_1,\vs_2\rangle}\sum_{i+j=2}\tbinom{2}{i}\Big\{\frac{1}{12}\partial_{ij}V(\vs_1,\vs_2)
\big[(\vs_1\cdot\nabla){\vu}^\pm\big]^i\notag\\&\cdot\big[(\vs_2\cdot\nabla){\vu}^\pm\big]^j+\frac{1}{4}\partial_{ij}V(\vs_1- \vp,\vs_2)\big[(\vs_1-\vp)\cdot(\nabla {\vu}^\pm)\big]^i\big[(\vs_2)\cdot(\nabla {\vu}^\pm)\big]^j\Big\}\notag\\&+\sum_{\vs}\big[U(\vphi+\vs-\vd)+U(\vphi+\vs-{\vp}-\vd)-U(\vs-\vd)-U(\vs-{\vp}-\vd)\big]\,\D x\,\D y.\label{PN b}
\end{align}

The way we divide the energy in each of the atomistic model and the PN model into two parts is that we put $AAA$, $AAB$ interactions (intralayer interactions) and interactions between $A+$, $A-$ and $B+$, $A-$ (inter-layer interactions) into the A part, and the rest of interactions are divided into the B parts.

Accordingly, we define $\gamma$-surface in $E_\text{PN}^A[\vu]$ as $\gamma_A(\vphi)$ and $\gamma$-surface in $E_\text{PN}^B[\vu]$ as $\gamma_B(\vphi)$:
\begin{align}
\gamma_A(\vphi)&:=\sum_{\vs}U(\vphi+\vs-\vd)+U(\vphi+\vs+{\vp}-\vd)-U(\vs-\vd)-U(\vs+{\vp}-\vd)\notag\\\gamma_B(\vphi)&:=\sum_{\vs}U(\vphi+\vs-\vd)+U(\vphi+\vs-{\vp}-\vd)-U(\vs-\vd)-U(\vs-{\vp}-\vd).\label{divide gamma}
\end{align}


\subsection{Assumptions}

Next we collect the assumptions that will be used in the proofs:

\textbf{A1} (weak inter-layer interaction): $\varepsilon\ll 1$.

\textbf{A2} (symmetry): $U(-\vx)=U(\vx)$ and $V(\vs_1,\vs_2)=V(-\vs_1,-\vs_2)$.

\textbf{A3} (regularity): $ U(\vx)\in C^5(\mathbb{R}\backslash\{0\})$ and $ V(\vs_1,\vs_2)\in C^5(\mathbb{R}^2\backslash\{|\vs_1|=0 \text{ or }|\vs_2|=0\})$.

\textbf{A4} (fast decay): There exist a constant $\theta>0$ such that:
\begin{align}
| U^{(k)}(\vx)|\le&|\vx|^{-k-4-\theta},\quad |\vx|>0,\notag\\
| V^{\vbeta}(\vs_1,\vs_2)|\le&|\vs_1|^{-\beta_1-6-\theta} |\vs_2|^{-\beta_2-6-\theta},\quad |\vs_1|,|\vs_2|>|\vp|,\notag
\end{align}
where $\vbeta=(\beta_1,\beta_2)$ and $V^{\vbeta}(\vs_1,\vs_2)=\partial_{\vs_1}^{\beta_1}\partial_{\vs_2}^{\beta_2}V(\vs_1,\vs_2)$.

\textbf{A5} (elasticity constant): $\alpha_1>0$ and $\alpha_2>0$.

\textbf{A6} ($\gamma$-surface): $\nabla^2\gamma(\vzero)$ is a positive-definite matrix and $\arg\min_{\vphi\in\mathbb{R}^2}=\sL$, where $\sL=L(\ve_i,\vo)$.

\textbf{A7} (stability division): $\bar{\vartheta}\ge \frac{1}{3}\vartheta$, where
\begin{align}
\bar{\vartheta}&=\min\left\{\inf_{\|\vf\|_{X_0}=1}\langle\delta^2 E_{\text{PN}}^A[\vv]\vf,\vf\rangle_0,\inf_{\|\vf\|_{X_0}=1}\langle\delta^2 E_{\text{PN}}^B[\vv]\vf,\vf\rangle_0\right\}\notag\\\vartheta&=\inf_{\|\vf\|_{X_0}=1}\langle\delta^2 E_{\text{PN}}[\vv]\vf,\vf\rangle_0,
\end{align}
where $\vv$ is the solution of Peierls--Nabarro model (Eqs.~(\ref{PN equation})) and $E_{\text{PN}}^A[\vv]$ (Eq.~(\ref{PN a})), $E_{\text{PN}}^B[\vv]$ (Eq.~(\ref{PN b})) are a division of $E_{\text{PN}}^A[\vv]$.

\textbf{A8} (small stability gap): $\Delta<\min\left\{\frac{1}{3},\frac{1}{3}\bar{\vartheta}\right\}$, where $\Delta=\max\{\Delta_A,\Delta_B\}$ and \begin{align}
\Delta_A&=\lim_{\varepsilon\to 0}\sup_{\|\vf\|_{X_\varepsilon}=1}\langle\delta^2 E_{\text{PN}}^A[\vzero]\vf_A,\vf_A\rangle_0-\langle\delta^2 E_{\text{a}}^A[\vzero]\vf,\vf\rangle_0\notag\\\Delta_B&=\lim_{\varepsilon\to 0}\sup_{\|\vf\|_{X_\varepsilon}=1}\langle\delta^2 E_{\text{PN}}^B[\vzero]\vf_B,\vf_B\rangle_0-\langle\delta^2 E_{\text{a}}^B[\vzero]\vf,\vf\rangle_0,
\end{align}
where $\vf_A$ and $\vf_B$ are the interpolations of complex lattice to be defined in Section 9.2.

For Assumption 1, as an example, $\varepsilon\thickapprox 0.0475\ll 1$ in the bilayer graphene, as shown in the  Appendix of Ref.~\cite{luo2018atomistic}.

Assumptions A2--A4 are satisfied by most pair potentials, such as the Lennard--Jones potential, the Morse potential, the Stillinger-Weber potential. As for Assumption A5--A6, the physical meaning is that perfect lattice without defects is a stable global minimizer of total energy.

For Assumption A7, we remark that $\vartheta>0$ (Proposition \ref{stability of PN}) describes the stability of the dislocation solution of Peierls--Nabarro model. If $\delta^2E_{\text{PN}}^A[\vu]=\delta^2E_{\text{PN}}^B[\vu]$, we obtain $\frac{1}{2}\vartheta=\bar{\vartheta}$. We notice the elastic energy parts in $\delta^2E_{\text{PN}}^A[\vu]$ and $\delta^2E_{\text{PN}}^B[\vu]$ are same: \begin{align}
&\sum_{\langle\vs_1,\vs_2\rangle}\partial_{ij}V(\vs_1- \vp,\vs_2)\big[(\vs_1-\vp)\cdot(\nabla {\vu}^\pm)\big]^i\big[(\vs_2)\cdot(\nabla {\vu}^\pm)\big]^j\notag\\=&\sum_{\langle\vs_1,\vs_2\rangle}\partial_{ij}V(-\vs_1- \vp,-\vs_2)\big[(-\vs_1-\vp)\cdot(\nabla {\vu}^\pm)\big]^i\big[(-\vs_2)\cdot(\nabla {\vu}^\pm)\big]^j\notag\\=&\sum_{\langle\vs_1,\vs_2\rangle}\partial_{ij}V(\vs_1+ \vp,\vs_2)\big[(\vs_1+\vp)\cdot(\nabla {\vu}^\pm)\big]^i\big[(\vs_2)\cdot(\nabla {\vu}^\pm)\big]^j\notag.
\end{align}
Here the first equality is due to $\vs_1$ and $\vs_2$ traverse all values of $\mathbb{L}^*$ and the second equality is due to $V(\vs_1,\vs_2)=V(-\vs_1,-\vs_2)$. Therefore the elastic constant in $\delta^2E_\text{PN}^A[\vu]$ is the equal to the elastic constant in $\delta^2E_\text{PN}^B[\vu]$. The misfit energy parts in $\delta^2E_{\text{PN}}^A[\vu]$ and $\delta^2E_{\text{PN}}^B[\vu]$ are different since $\vd\not=\vzero$ (the vector between the centers of two layers), i.e., $\gamma''_A(\vphi)\not=\gamma''_B(\vphi)$ (Eqs.~\eqref{divide gamma}). However, we notice that \begin{align}&\gamma''_A(\vphi)-\frac{1}{2}\gamma''(\vphi)\notag\\=&\sum_{\vs}\frac{1}{2}\left[U''(\vphi+\vs+{\vp}-\vd)-U''(\vphi+\vs-{\vp}-\vd)\right]\notag\\\approx&\sum_{\vs}\vp\cdot U^{(3)}(\vphi+\vs-\vd).\end{align}
Due to Assumption A4, $U^{(3)}$ is smaller than $U''$. Therefore, $\gamma''_A(\vphi)$ and $\gamma''_B(\vphi)$ are close to $\frac{1}{2}\gamma''(\vphi)$ and $\delta^2E_\text{PN}^A[\vu]$ and $\delta^2E_\text{PN}^B[\vu]$ are close to $\frac{1}{2}\delta^2E_\text{PN}[\vu]$. Therefore, Assumption A7 is reasonable.

For Assumption A8, we remark that $\Delta_A$ ($\Delta_B$) characterizes the stability gap between atomistic model $\delta^2E_\text{a}^A[\vzero]$ ($\delta^2E_\text{a}^B[\vzero]$) and PN model $\delta^2E_\text{PN}^A[\vzero]$ ($\delta^2E_\text{PN}^B[\vzero]$) at the perfect lattice. We also provide an explicit formula for $\delta^2E_\text{a}^A[\vzero]$ in Lemma \ref{Delta}. In  Remark \ref{Ipni}, we show $\Delta_A=0$ and Assumption 7 holds when we consider nearest neighbor potential.

\section{Main results}
Based on the energies of the PN model \eqref{re PN} and  the atomistic model \eqref{at red}, we obtain the following two Euler--Lagrange equations.

The Euler--Lagrange equation of the PN model reads as
\begin{equation}
\begin{cases}
 \delta E_{\text{PN}}[\vu] =0      & \\
 \vu=(u,0) & \\
{\displaystyle \lim_{x\to-\infty}}u^{\perp}(x,y)=0, \ \ {\displaystyle \lim_{x\to+\infty}}u^{\perp}(x,y)=1,& \\
u^+(0,y)=\frac{1}{4},&
\end{cases}\label{PN equation}
\end{equation}

The Euler--Lagrange equation of the atomistic model reads as
\begin{equation}
\begin{cases}
\delta E_{\text{a}}[\vu] =0  &\\\vu=(u,0) &\\u^{\kappa+} _{\vs}=-u^{\kappa-} _{\vs}&     \\
u^{\kappa\pm} _{\vs}+u^{\kappa\pm} _{-\vs}=\pm\frac{1}{2}&\\
u\text{ satisfies Atomistic Dislocation Condition (ADC)},& \\
{\displaystyle \lim_{s\cdot \ve_1\to+\infty}}(u_{\vs}^{\kappa})^{\perp}=1, \ \ {\displaystyle \lim_{\vs\cdot \ve_1\to-\infty}}(u_{\vs}^{\kappa})^{\perp}=0,&
\end{cases}\label{Atomistic equation}
\end{equation}
where $\kappa=A,B$ and Atomistic Dislocation Condition (ADC) will be defined in Section~\ref{sec:adc}.

\newtheorem{mythm}{Theorem}
\begin{mythm}\label{Theorem 1}
(Existence of PN model) If Assumption A2--A5 holds, the PN problems \eqref{PN equation} has a unique solution $\vv=(v,0)\in S_0$, $v=\left\{v^+,v^-\right\}$, $v^+=-v^-$ and $v^\pm(x,y)+v^\pm(-x,-y)=\pm\frac{1}{2}$, which is also the $X_0$--global minimizer of the total energy of (\ref{re PN}). Furthermore, we have $v^+_y(x,y)=0,v^+_x(x,y)>0$ for all $x,y$ and $v^+(\cdot,y)\in W^{5,\infty}(\mathbb{R})$ and $v^+(\cdot,y)\in W^{4,1}(\mathbb{R})$ for all $y$.
\end{mythm}
\begin{mythm}\label{thm:Atom}(Existence of atomistic model and convergence) If Assumption A1--A6 holds, there exist an $\varepsilon_0$ such that for any $0<\varepsilon<\varepsilon_0$, the atomistic model problem \eqref{Atomistic equation} has a unique $\vv^\varepsilon\in S_\varepsilon$, which is also an $X_\varepsilon$--local minimizer of the total energy (\ref{at red}). Furthermore, $||\vv^\varepsilon-\vv||_{X_\varepsilon}\le C\varepsilon^2$, where $v$ is the solution of PN model.
\end{mythm}

\section{Proof of Theorem 1}
In this section, we prove Theorem \ref{Theorem 1}.
  We first prove that $v^+=-v^-$ and $\partial_yv^\pm=0$ for all  $(x,y)\in\mathbb{R}^2$. Therefore, we can reduce the two-dimensional problem to a one-dimensional problem, and then the solution existence of the two dimensional PN model follows the  existence  of the one-dimensional  PN model proved in Ref.~\cite{luo2018atomistic}.

Since the solution of the PN model (\ref{PN equation})  is the global minimizer of the associated energy, we divide the one-step minimization into a two-step minimization:

\textbf{(1)} Given $\phi\in\Phi$,  find $u_\phi=\left\{u^+_\phi,u^-_\phi\right\}$, where $u_\phi\in S_0$ with $u^\bot_\phi=\phi$ such that $$
E_{\mathrm{elas}}\left[u_{\phi}\right]=\inf _{u\in S_0, u^{\perp}=\phi} E_{\mathrm{elas}}[u]
$$ and we define $E^{II}_{\text{elas}}[\phi]=\inf _{u\in S_0, u^{\perp}=\phi} E_{\mathrm{elas}}[u]$.

\textbf{(2)} Find $\phi^*\in \Phi$ such that $$
E_{\mathrm{PN}}^{I I}\left[\phi^{*}\right]=\inf _{\phi^*\in \Phi} E_{\mathrm{PN}}^{I I}[\phi]=E_{\mathrm{elas}}^{I I}[\phi]+E_{\mathrm{mis}}[\phi].
$$
Here $$\Phi:=\left\{\phi\in C^5(\mathbb{R}^2)\mid\left\{\phi-\frac{2}{\pi}\arctan x+1,-\phi+\frac{2}{\pi}\arctan x-1\right\}\in X_0\right\}.$$

Similar to Proposition 1 in Ref.~\cite{luo2018atomistic}, we have the follow lemma:
\begin{mylem}\label{u=phi}
Suppose that $E_{\text{PN}}[u]<\infty$. Then the two-step minimization problem is equivalent to the one-step minimization problem.
\end{mylem}
\begin{proof}
The rigorous proof is given in \cite[Proposition 1]{luo2018atomistic}.
\end{proof}

We also need the following two lemmas for the proof of Theorem 1. Lemma~\ref{ODE T} is from  Ref.~\cite{pinchover2005introduction}, and the proof of Lemma~\ref{gamma re}  is similar as that of \cite[Lemma 3]{luo2018atomistic}.

\begin{mylem}[\cite{pinchover2005introduction}]\label{ODE T} Suppose that the curve $\gamma:(x,y,z)=(f(s),g(s),h(s))$, $s\in \mathbb{R}$, is smooth, and $f'^2+g'^2\not=0$, if the  Jacobian determinant in $$P_0=(x_0,y_0,z_0)=(f(s_0),g(s_0),h(s_0))$$ is non-zero, i.e.,\[ J= \begin{vmatrix}
f'(s_0)&g'(s_0)\\
a(x_0,y_0,z_0)&b(x_0,y_0,z_0)
\end{vmatrix}\not=0.\] And $a(P_0),b(P_0),c(P_0)$ are smooth near $\gamma$. Then the Cauchy initial value problem:$$\left\{\begin{array}{l}{a(x,y,u)\partial_x u+b(x,y,u)\partial_y u=c(x,y,u)}\\ {u(f(s),g(s))=h(s)} .\end{array}\right. $$has a unique solution near $P_0$.
\end{mylem}

\begin{mylem}\label{gamma re}
Suppose $U(\vx)=U(-\vx)$, then we obtain \begin{align}
&(1)~~\gamma(\vxi) \geq \frac{1}{4} \gamma_{xx}(\vzero ) |\vxi|^{2},|\vxi| \leq C
,\notag\\&(2)~~\gamma(\ve_1+\vxi)=\gamma(\vxi),\notag\\&(3)~~ \gamma(\vxi_1)=\gamma(\vxi_2) \text{ when } \vxi_1 , \vxi_2 \text{ are symmetric with respect to the y-axis},
\end{align}
where $C$ depends on $\gamma_{xx}(\vzero)$ and $\ve_1=(1,0)$.
\end{mylem}

\begin{proof}[Proof of Theorem 1]
Due to Lemma \ref{no xy}, we can rescale $y$ to $\bar{y}$:
\begin{equation}
    \bar{y}=\sqrt{\frac{\alpha_1}{\alpha_2}}y,\label{change constant}
\end{equation}
where $\alpha_1$ and $\alpha_2$ are defined in Lemma \ref{no xy}. To simplify the notation, we denote $\bar{y}$ as $y$ and $\gamma(\xi):=\gamma\big((\xi,0)\big)$ in this proof. Furthermore, we notice \eqref{change constant} dose not change the boundary condition of PN model. Thus we define $$\alpha:=\frac{1}{2}\sqrt{\alpha_1\alpha_2}. $$In the PN model, the elastic energy is that $$E_{\text{elas}}=\lim_{{R}\to +\infty}\frac{1}{2R}\int_{-R}^{R}\int_{{\mathbb{R}}}\frac{1}{2}\alpha(|\nabla u^{+}|^2+|\nabla u^{-}|^2)\,\D x \,\D y.$$

We then divide the one-step minimization problem into two-step minimization problem by Lemma \ref{u=phi}. For any $\phi\in \Phi$, we have
\begin{align}
&\arg \min _{u \in S_{0}, u^{\perp}=\phi} E_{\mathrm{elas}}[u]\notag\\=&\arg \min _{u \in S_{0}} \lim_{{R}\to +\infty}\frac{1}{2R}\int_{-R}^{R}\int_{{\mathbb{R}}} \frac{1}{2} \alpha\left(\left|\nabla u^{+}\right|^{2}+\left|\nabla u^{+}-\nabla \phi\right|^{2}\right) \,\D x \,\D y.
\end{align}
Due to symmetry of quadratic functions,we have $$\arg \min _{u \in S_{0}, u^{\perp}=\phi} E_{\mathrm{elas}}[u]=\left\{\frac{1}{2}\phi,-\frac{1}{2}\phi\right\}.$$

By Lemma \ref{u=phi}, we further minimize the following energy $E^{II}_{\text{PN}}[\phi]$:
\begin{align}
E_{\mathrm{PN}}^{I I}[\phi]&=\lim_{{R}\to +\infty}\frac{1}{2R}\int_{-R}^{R}\int_{{\mathbb{R}}}\left(\frac{1}{4} \alpha|\nabla \phi|^{2}+\gamma(\phi)\right) \,\D x \,\D y\notag\\
&\ge\lim_{{R}\to +\infty}\frac{1}{2R}\int_{-R}^{R}\int_{{\mathbb{R}}}|\nabla \phi(x,y)| \sqrt{\alpha \gamma(\phi(x,y))}\,\D x \,\D y\notag\\
&=\lim_{{R}\to +\infty}\frac{1}{2R}\int_{-R}^{R}\int_{{\mathbb{R}}}\left|\nabla \Gamma(\phi(x,y))\right|\,\D x \,\D y\notag\\
&\geq \Big|\lim_{{R}\to +\infty}\frac{1}{2R}\int_{-R}^{R}\int_{{\mathbb{R}}}\nabla \Gamma(\phi(x,y))\,\D x \,\D y\Big|\notag\\
&=\Big|\Big(\lim_{{R}\to +\infty}\frac{1}{2R}\int_{-R}^{R} \big[\lim _{x \rightarrow+\infty} \Gamma(\phi(x,y))-\lim _{x \rightarrow-\infty} \Gamma\left(\phi(x,y)\right)\big]\,\D y,\notag\\&\quad~ \lim_{{R}\to +\infty}\frac{1}{2R}\int_{{\mathbb{R}}}\big[\Gamma(\phi(x,R))-\Gamma(\phi(x,-R))\big]\,\D x\Big)\Big|\notag\\&=\Big|\Big(\int_{0}^{1} \sqrt{\alpha \gamma(\eta)} \mathrm{d} \eta,\lim_{{R}\to +\infty}\frac{1}{2R}\int_{{\mathbb{R}}}\big[\Gamma(\phi(x,R))-\Gamma(\phi(x,-R))\big]\,\D x\Big)\Big|\notag\\&\ge\int_{0}^{1} \sqrt{\alpha \gamma(\eta)} \mathrm{d} \eta ,
\end{align}
where $\Gamma(\xi)=\int_{0}^{\xi} \sqrt{\alpha \gamma(\eta)} \mathrm{d} \eta \text { for } \xi \in \mathbb{R}$. Note that the first inequality holds if and only if $\frac{1}{2} \sqrt{\alpha}|\nabla \phi|=\sqrt{\gamma \circ \phi}$. The second inequality  holds if and only if there exists a constant $\lambda$, s.t $\partial_x g=\lambda \partial_y g$ or $\partial_y g=0$ in $\mathbb{R}\times (-R,R)$ for all $R$, where $g=\Gamma(\phi(x,y))$, and $g(0,y)=\int_{0}^{\frac{1}{2}} \sqrt{\alpha \gamma(\eta)} \mathrm{d} \eta$.

Due to Lemma \ref{ODE T}, we can solve this partial differential equation of first order by the method of characteristics except $\partial_y g=0$.

(1) $\lambda \partial_x g=\partial_y g,\lambda \not=0$, i.e.,  \[ J= \begin{vmatrix}
0&1\\
\lambda&-1
\end{vmatrix}\not=0.\]Thus this problem is equivalent to $$\left\{\begin{array}{l}{\frac{\,\D x}{\,\D t}=\lambda,\quad \frac{\,\D y}{\,\D t}=-1,\quad \frac{\,\D z}{\,\D t}=0}\\ {(x,y,z)|_{t=0}=(0,s,\int_{0}^{\frac{1}{2}} \sqrt{\alpha \gamma(\eta)} \mathrm{d} \eta)} .\end{array}\right. $$Thus we get $g\equiv \int_{0}^{\frac{1}{2}} \sqrt{\alpha \gamma(\eta)} \mathrm{d} \eta$ near $\{0\}\times (-R,R)$, which is contradictory because
 the dislocation solution near the $x=0$ is not the constant solution.

(2) $\partial_x g=0,\lambda=0$, i.e.,\[ J= \begin{vmatrix}
0&1\\
1&0
\end{vmatrix}\not=0.\]Thus we can also have $g\equiv \int_{0}^{\frac{1}{2}} \sqrt{\alpha \gamma(\eta)} \mathrm{d} \eta$ near $\{0\}\times (-R,R)$, which is contradictory.

(3) $\partial_y g=0,\lambda=0$, i.e.,\[ J= \begin{vmatrix}
0&1\\
0&1
\end{vmatrix}=0,\]which can let second inequality hold its equality. Hence $g=\int_{0}^{\phi(x,y)} \sqrt{\alpha \gamma(\eta)} \mathrm{d} \eta $, the $\partial_y g=0$ is equal to $\partial_y \phi(x,y)=0$ near $\{0\}\times (-R,R)$. Due to the arbitrariness of $R$ and extension theorem, we have $\partial_y \phi(x,y)=0$ in $\mathbb{R}^2$.

In addition, the third inequality hold its equality if and only if $$\lim_{{R}\to +\infty}\frac{1}{2R}\int_{{\mathbb{R}}}[\Gamma(\phi(x,R))-\Gamma(\phi(x,-R))]\,\D x=0,$$ which holds when $\partial g=0$.

In  summary, we have $E_{\mathrm{PN}}^{I I}[\phi]\ge\int_{0}^{1} \sqrt{\alpha \gamma(\eta)} \mathrm{d} \eta $, and the equality holds if and only if \begin{equation}
\left\{\begin{array}{l}{\frac{1}{2} \sqrt{\alpha}|\nabla \phi|=\sqrt{\gamma \circ \phi}}\\ {\partial_y \phi(x,y)=0} .\end{array}\right. \label{equation ode}
\end{equation}
We have shown that $\phi(x,y)$ only depends on $x$ when $u$ is the $X_0$-global minimizer of the energy of $E_{\text{PN}}$ and $\partial_x\phi\ge 0$ is obvious. Thus we can only consider the following equation:\begin{equation}\frac{\,\D \phi}{\,\D x}=\frac{2}{\sqrt{\alpha}} \sqrt{\gamma(\phi)}, \quad \phi(0)=\frac{1}{2}.\label{Ode}\end{equation}

Now we have  reduced the two dimensional problem to a one dimensional problem. Then Theorem 1 follows Lemma \ref{gamma re} and \cite[Theorem 1]{luo2018atomistic}.
\end{proof}

In the proof of Theorem \ref{Theorem 1}, we  a two-dimensional problem to a one-dimensional problem. Thus we can get the stability of  the solution of the PN model from \cite{luo2018atomistic}:
\newtheorem{mypro}{Proposition}
\begin{mypro}\label{stability of PN}
 Let $\vv$ be the dislocation solution of PN model, then there exist $\vartheta>0$ such that for any $\vf\in X_0$, we have \begin{equation}
\left \langle \delta^2E_{\text{PN}}[\vv]\vf,\vf \right \rangle_0\ge\vartheta
||\vf||_{X_0}^2.\label{pro 1}
\end{equation}
\end{mypro}
\begin{proof}
	Thanks to \cite[Proposition 3]{luo2018atomistic}, for all $y\in\mathbb{R}$, $v(\cdot,y)$ satisfies:\begin{align}
	&\int_{\mathbb{R}}\alpha_1 (\partial_x f^+(x,y))^2+\alpha_2(\partial_y f^+(x,y))^2+2\gamma''(2v)(f^+(x,y))^2\,\D x\notag\\\ge&\vartheta(y)\left[\int_{\mathbb{R}}\alpha_1 (\partial_x f^+(x,y))^2+\alpha_2(\partial_y f^+(x,y))^2+\frac{8\sqrt{3}}{3}(f^+(x,y))^2\,\D x\right],\label{one dim}
	\end{align}
	where $\vartheta(y)$ is the constant depending on $y$. Due to uniform of $v$ in $y$-direction, $\vartheta(y)$ is also uniform in $y$-direction. Thus we calculate the average value of both side in (\ref{one dim}) along $y$-direction and obtain (\ref{pro 1}).
\end{proof}

\section{Consistency of PN Model}

In this section, we prove the consistency of PN model.

For simplicity of notations, we introduce some notations:
\begin{align}
U_{k, \vs}&:=\sup _{|\vxi-\vs| \leq 1}\left|\nabla^{k} U(\vxi)\right|, \quad \vs \in \mathbb{L},\label{U max}\\
V_{\vbeta,\vs_1,\vs_2}&:=\sup _{|\vxi_i-\vs_i| \leq 1}\left|\nabla^{\vbeta} V(\vxi_1,\vxi_2)\right|, \quad i=1,2,~~\vbeta \in \mathbb{Z}^2,\\\notag
\vv_{k,\vs_1,\vs_2}&:=\sup_{|\vx-\vs_1|<|\vs_2|}\left|\nabla^{\vbeta} v(x,y)\right|, \quad |\vbeta|=k,
\end{align}
where $\vv=(v,0)$.
\begin{mylem}\label{A4 come}
	Suppose that Assumptions A2--A4 hold and $k=1,2\cdots,5$. Then there exist constants $H_1, H_2$ such that\begin{equation}
	\sum_{\vs}|\vs|^{k+1}U_{k,\vs}\le H_1,\quad\sum_{\langle\vs_1,\vs_2\rangle}|\vs_1|^{a_1}|\vs_2|^{a_2}V_{\valpha,\vs_1,\vs_2}\le H_2,\label{le1}
	\end{equation}
	where $a_i\le |\beta_i|+4$ and $\beta_i$ is defined in A4 for $i=1,~2$.
\end{mylem}
\begin{proof}
	Without loss of generality, we just prove
	$$\sum_{\langle\vs_1,\vs_2\rangle}|\vs_1|^{a_1}|\vs_2|^{a_2}V_{\valpha,\vs_1,\vs_2}\le H_2.$$
	
	As for A4, we have \begin{align}
	\sum_{\langle\vs_1,\vs_2\rangle}|\vs_1|^{a_1}|\vs_2|^{a_2}V_{\valpha,\vs_1,\vs_2}\le& \sum_{\langle\vs_1,\vs_2\rangle}|\vs_1|^{a_1}|\vs_2|^{a_2} |\vs_1|^{-a_1-2-\theta}|\vs_2|^{-a_2-2-\theta}\notag\\\le&\sum_{\vs_1}|\vs_1|^{-2-\theta}\sum_{\vs_2}|\vs_2|^{-2-\theta}\le H_2.\notag
	\end{align}
\end{proof}
\begin{mylem}\label{s come}
	Suppose that A1--A5 hold, then for $k\le 4$, we have
	$$\varepsilon^2\lim_{R\to+\infty}\frac{1}{2R}\sum_{\vx\in\mathbb{L}_{R}}\vv_{k,\vx,\vs}\le C|\vs|,$$
	where $\vv$ is the dislocation solution in Theorem \ref{Theorem 1}.
\end{mylem}
\begin{proof}
	Since in Theorem \ref{Theorem 1}, we have proved $\partial_y v^\pm=0$ and $\vv=(v,0)$. Thus two dimensional problem can be reduced into one dimensional problem. Thanks to \cite[Lemma 7]{luo2018atomistic}, we have proved Lemma \ref{s come}.
\end{proof}

In the following analysis of the paper, the constant $C$ may be different from line to line.

\begin{mypro}\label{consistency pro}
	Suppose that Assumptions A1--A6 hold. Let $\vv$ be the dislocation solution of the PN model in Theorem 1, then there exist $C$ and $\varepsilon_0$, when $0<\varepsilon<\varepsilon_0$ and $\vf\in X_\varepsilon$ we have
	\begin{equation}
	\Big|\langle\delta E_{\text{a}}[\vv]-\delta E_{\text{PN}}[\vv],\vf\rangle_\varepsilon\Big|=\Big|\langle\delta E_{\text{a}}[\vv],\vf\rangle_\varepsilon\Big|\le C\varepsilon^2\|\vf\|_{X_\varepsilon},\label{consistency}
	\end{equation}
	where $C$ is independent on $\varepsilon$.
\end{mypro}
\begin{proof}
	As for the $\Big|\langle\delta E_{\text{a}}[\vv]-\delta E_{\text{PN}}[\vv],\vf\rangle_\varepsilon\Big|$, there are two parts, elastic energy parts and misfit energy parts. Furthermore, for the elastic energy parts, it can be divided into four parts, $AAA$, $AAB$, $ABB$ and $BBB$. For the $AAA$ and $BBB$ parts, they are simple lattice cases, which has been proved in \cite{luo2018atomistic} (second order  accuracy). The remaining question is to estimate $AAB$ and $ABB$ parts:\begin{align}
	&R_\text{elas}\notag\\=&\lim_{R\to+\infty}\frac{1}{4R}\sum_{\vx\in\mathbb{L}_{R}}\sum_{\langle\vs_1,\vs_2\rangle}\Big\{\left[\partial_{1}V(\vp+\vs_1+\varepsilon D_{\vs_1}^{\vp} \vv^{A\pm}_{\vx},\vs_2+{\varepsilon D_{\vs_2}\vv^{A\pm}_{\vx}})\right]\notag\\&\varepsilon D_{\vs_1}^{\vp} \vf^{A\pm}_{\vx}+\left[\partial_{2}V(\vp+\vs_1+\varepsilon D_{\vs_1}^{\vp} \vv^{A\pm}_{\vx},\vs_2+{\varepsilon D_{\vs_2}\vv^{A\pm}_{\vx}})\right]\varepsilon D_{\vs_2}\vf^{A\pm}_{\vx}\notag\\&\left[\partial_{1}V(-\vp+\vs_1+\varepsilon D_{\vs_1}^{-\vp} \vv^{B\pm}_{\vx},\vs_2+{\varepsilon D_{\vs_2}\vv^{B\pm}_{\vx}})\right]\varepsilon D_{\vs_1}^{-\vp} \vf^{B\pm}_{\vx}\notag\\&+\left[\partial_{2}V(-\vp+\vs_1+\varepsilon D_{\vs_1}^{-\vp} \vv^{B\pm}_{\vx},\vs_2+{\varepsilon D_{\vs_2}\vv^{B\pm}_{\vx}})\right]\varepsilon D_{\vs_2}\vf^{B\pm}_{\vx}\Big\}\notag\\&+\frac{\varepsilon^2}{2}\sum_{\langle\vs_1,\vs_2\rangle}\sum_{i+j=2}\tbinom{2}{i}\partial_{ij}V(\vs_1+ \vp,\vs_2)\big[(\vs_1+\vp)\cdot(\nabla )\big]^i(\vs_2\cdot\nabla)^j \vv^{A\pm}_{\vx}\vf^{A\pm}_{\vx}\notag\\&+\frac{\varepsilon^2}{2}\sum_{\langle\vs_1,\vs_2\rangle}\sum_{i+j=2}\tbinom{2}{i}\partial_{ij}V(\vs_1+ \vp,\vs_2)\big[(\vs_1+\vp)\cdot(\nabla )\big]^i(\vs_2\cdot\nabla)^j \vv^{B\pm}_{\vx}\vf^{B\pm}_{\vx}.
	\end{align}
	The first four terms are from $\delta E_{\text{a}}[\vv]$ and the last two terms is from $\delta E_{\text{PN}}[\vv]$.
	
	As for the misfit energy parts, there are four parts, $A^+A^-$ (the inter-layer interactions between the atoms $A$ in the upper layer and the atoms $A$ in the lower layer), $A^+B^-$, $B^+A^-$ and $B^+B^-$. $A^+A^-$ and $B^+B^-$ are the simple lattice cases. Therefore we restrict our attention on $A^+B^-$ and $B^+A^-$. We estimate $B^+A^-$ only:\begin{align}
	&R_\text{mis}\notag\\=&\lim_{R\to+\infty}\frac{\varepsilon^2}{2R}\sum_{\vx\in\mathbb{L}_{R}}\sum_{\vs}\Big[U'(\vs+\vp-\vd+\vv_{\vx+\vs}^{B+}-\vv_{\vx}^{A-})(\vf_{\vx+\vs}^{B+}-\vf_{\vx}^{A-})\notag\\&+U'(\vs+\vp-\vd+\vv_{\vx}^{B+}-\vv_{\vx-\vs}^{A-})(\vf_{\vx}^{B+}-\vf_{\vx-\vs}^{A-})\notag\\&-U'(\vs+\vp-\vd+\vv_{\vx}^{A+}-\vv_{\vx}^{A-})(\vf_{\vx}^{A+}-\vf_{\vx}^{A-})\notag\\&-U'(\vs+\vp-\vd+\vv_{\vx}^{B+}-\vv_{\vx}^{B-})(\vf_{\vx}^{B+}-\vf_{\vx}^{B-})\Big]
	\end{align}
	The first two terms are from $\delta E_{\text{a}}[\vv]$ and the last two terms are from $\delta E_{\text{PN}}[\vv]$.
	
	1. \textbf{$R_\text{elas}$:}
	
	The calculation is too hard if we calculate $R_\text{elas}$ directly as \cite{luo2018atomistic}. However, we just need to estimate the order of $\varepsilon$. In other words, we show $\varepsilon,~\varepsilon^2$ and $\varepsilon^3$ order terms disapper in $R_\text{elas}$. Without loss the generality, we consider one term from $AAB$ interactions in atomistic model: \begin{align}
	&\lim_{R\to+\infty}\frac{1}{4R}\sum_{\vx\in\mathbb{L}_{R}}\sum_{\langle\vs_1,\vs_2\rangle}\notag\\&\left[\partial_{1}V(\vp+\vs_1+\varepsilon D_{\vs_1}^{\vp} \vv^{A\pm}_{\vx},\vs_2+{\varepsilon D_{\vs_2}\vv^{A\pm}_{\vx}})\right]\varepsilon D_{\vs_1}^{\vp} \vf^{A\pm}_{\vx}\notag.\label{consider term}
	\end{align}
	
	Due to Lemma \ref{no xy}, we notice the pairwise $AB$ has rotational symmetry:
	\begin{align}
	&\lim_{R\to+\infty}\frac{1}{4R}\sum_{\vx\in\mathbb{L}_{R}}\sum_{\langle\vs_1,\vs_2\rangle}\notag\\&\left[\partial_{1}V(\vp+\vs_1+\varepsilon D_{\vs_1}^{\vp} \vv^{A\pm}_{\vx},\vs_2+{\varepsilon D_{\vs_2}\vv^{A\pm}_{\vx}})\right] \vf^{A\pm}_{\vx}\notag\\=&\lim_{R\to+\infty}\frac{1}{12R}\sum_{\vx\in\mathbb{L}_{R}}\sum_{\langle\vs_1,\vs_2\rangle}\Big[\partial_{1}V(\vp+\vs_1+\varepsilon D_{\vs_1}^{\vp} \vv^{A\pm}_{\vx},\vs_2+{\varepsilon D_{\vs_2}\vv^{A\pm}_{\vx}})\notag\\&+\partial_{1}V(\vp+\vs_{11}+\varepsilon D_{\vs_{11}}^{\vp} \vv^{A\pm}_{\vx},\vs_{21}+\varepsilon D_{\vs_{21}}\vv^{A\pm}_{\vx})\notag\\&+\partial_{1}V(\vp+\vs_{12}+\varepsilon D_{\vs_{12}}^{\vp} \vv^{A\pm}_{\vx},\vs_{22}+\varepsilon D_{\vs_{22}}\vv^{A\pm}_{\vx})\Big] \vf^{A\pm}_{\vx}
	\end{align}
	where the definitions of $\vs_{11}$, $\vs_{12}$, $\vs_{21}$ and $\vs_{22}$ are from Lemma \ref{no xy} (by rotation of $2\pi/3$ and $4\pi/3$ for $\vs_1$ and $\vs_2$). Thanks to the $\vf^{A+}_{\vx}=-\vf^{A-}_{\vx}$ and $\vv^{A+}_{\vx}=-\vv^{A-}_{\vx}$, there is no even order terms in the Taylor expansion at $(\vp+\vs_1,\vs_2)$:\begin{align}
	&\lim_{R\to+\infty}\frac{1}{4R}\sum_{\vx\in\mathbb{L}_{R}}\sum_{\langle\vs_1,\vs_2\rangle}\notag\\&\left[\partial_{1}V(\vp+\vs_1+\varepsilon D_{\vs_1}^{\vp} \vv^{A\pm}_{\vx},\vs_2+{\varepsilon D_{\vs_2}\vv^{A\pm}_{\vx}})\right]\varepsilon D_{\vs_1}^{\vp} \vf^{A\pm}_{\vx}\notag\\=&\lim_{R\to+\infty}\frac{1}{6R}\sum_{\vx\in\mathbb{L}_{R}}\sum_{\langle\vs_1,\vs_2\rangle}\Bigg\{\partial_{20}V(\vp+\vs_1,\vs_2)\varepsilon( D_{\vs_{1}}^{\vp} \vv^{A+}_{\vx}+ D_{\vs_{11}}^{\vp} \vv^{A+}_{\vx}+ D_{\vs_{12}}^{\vp} \vv^{A+}_{\vx})\notag\\&+\partial_{11}V(\vp+\vs_1,\vs_2)(\varepsilon D_{\vs_{2}} \vv^{A+}_{\vx}+\varepsilon D_{\vs_{21}} \vv^{A+}_{\vx}+\varepsilon D_{\vs_{22}} \vv^{A+}_{\vx})\notag\\&+\varepsilon^3\sum_{\vt}\partial_{\vt}V(\vxi)\tbinom{3}{t_1}\Big[(D_{\vs_{1}}^{\vp} \vv^{A+}_{\vx})^{t_1}(D_{\vs_{2}} \vv^{A+}_{\vx})^{t_2}+(D_{\vs_{11}}^{\vp} \vv^{A+}_{\vx})^{t_1}(D_{\vs_{21}} \vv^{A+}_{\vx})^{t_2}\notag\\&+(D_{\vs_{21}}^{\vp} \vv^{A+}_{\vx})^{t_1}(D_{\vs_{22}} \vv^{A+}_{\vx})^{t_2}\Big]\Bigg\}\vf^{A+}_{\vx},\label{change relas}
	\end{align}
	where $\vxi$ is a four dimensional vector, $$t\in\left\{\vt=(t_1+1,t_2)\mid t_1,t_2\in\mathbb{N},t_1\ge0,t_2\ge0,t_1+t_2=3\right\}.$$ Due to the calculate in Lemma \ref{no xy} (the symmetric of $\vs_i,~\vs_{i1},~\vs_{i2}$), we know $\varepsilon$ terms vanish in (\ref{change relas}). Furthermore, $\varepsilon^2$ terms vanish in $R_\text{elas}$ since the formula of elastic energy in PN model is from the second order terms of the atomistic model (the definition of (\ref{elas no})). In other words, the definition of elastic energy in Peierls--Nabarro model is to collect all the second order terms of intralayer interactions in atomistic model. Hence the second-order terms in $R_\text{elas}$ disappears naturally. The remained problem is to show $\varepsilon^3$ terms vanish in the $R_\text{elas}$. For $O(\varepsilon^4)$ terms, the sum of them is $O(\varepsilon^2)$ as (\ref{o4too2}). But by direct calculation, the the symmetric of $\vs_i,~\vs_{i1},~\vs_{i2}$ can not derives the disappearance of $\varepsilon^3$. However, for each $\varepsilon^3$ term in (\ref{consider term}), we are able to find the similar terms in the $ABB$ part. Now we restrict our attention to the $O(\varepsilon^3)$ terms in \begin{equation}\partial_{20}V(\vp+\vs_1,\vs_2)(\varepsilon D_{\vs_{1}}^{\vp} \vv^{A+}_{\vx}+\varepsilon D_{\vs_{11}}^{\vp} \vv^{A+}_{\vx}+\varepsilon D_{\vs_{12}}^{\vp} \vv^{A+}_{\vx})\vf^{A+}_{\vx},\end{equation}which can be denoted as \begin{equation}
	\partial_{20}V(\vp+\vs_1,\vs_2)|\vs_1|^3\varepsilon^3(\vh\cdot\nabla^3\vv^{A+}_{\vx})\vf^{A+}_{\vx},
	\end{equation}
	where $\vh$ is a constant vector and $\vh\cdot\nabla^3\vv^{A+}_{\vx}$ is a two-dimensional vector. In addition, in the $ABB$ part, the similar terms is \begin{equation}
	\partial_{20}V(\vp+\vs_1,\vs_2)|\vs_1|^3\varepsilon^3(-\vh\cdot\nabla^3\vv^{B+}_{\vx})\vf^{B+}_{\vx}.
	\end{equation}
	The reason why the direction of corresponding item is $-\vh$ is shown in Figure \ref{h}. As in Figure \ref{h}, three red lines are interactions in $AAB$ (\ref{consider term}) part and three blue lines are interactions in $ABB$ part and they are center symmetry of $\vx+\frac{1}{2}\vp$ from each other.
	\begin{figure}[h]
		\centering
		\scalebox{0.35}{\includegraphics{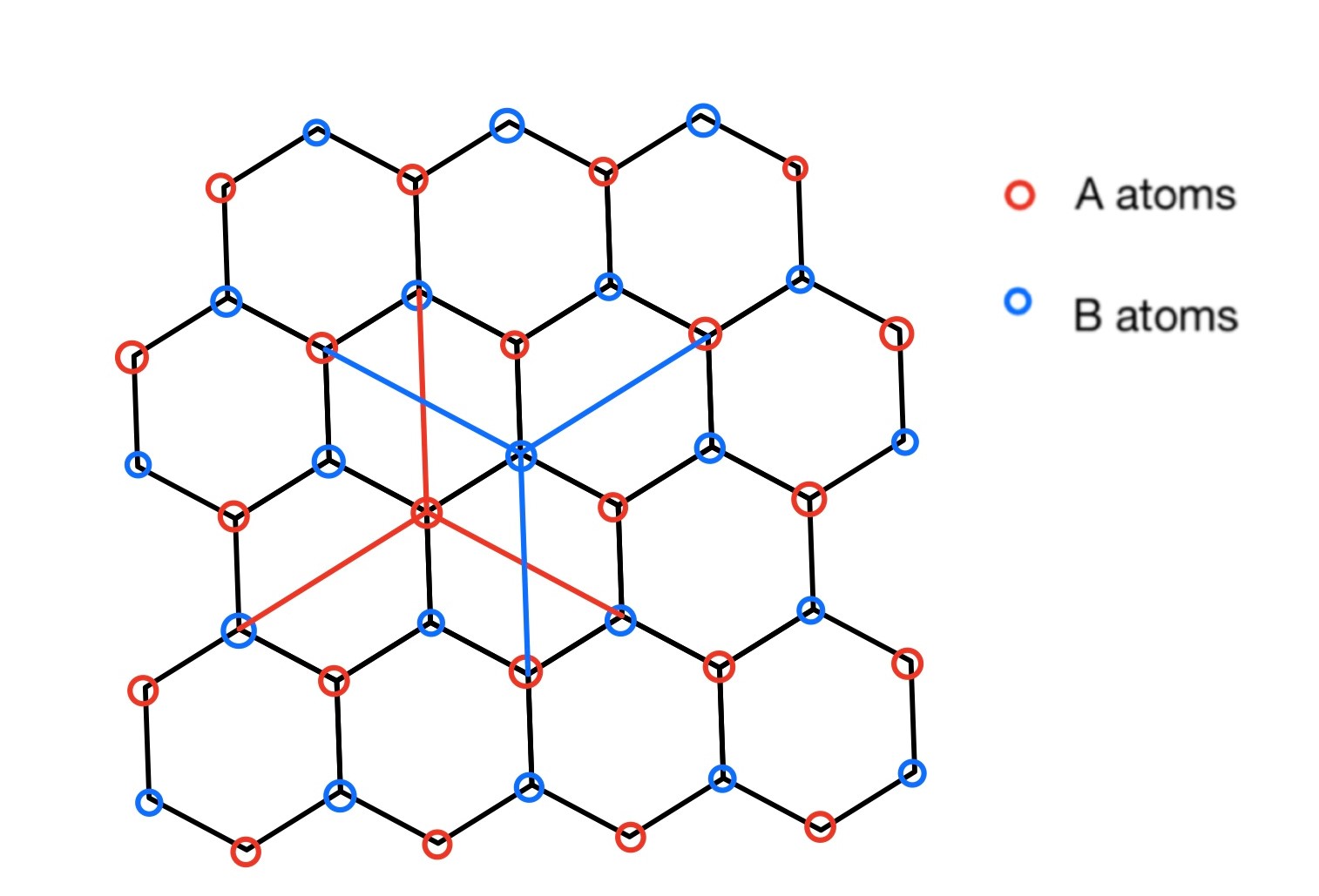}}
		\caption{Similar terms in the interactions in $ABB$ part.}
		\label{h}
	\end{figure}
	Combining those two terms, we have \begin{align}
	&\left|\partial_{20}V(\vp+\vs_1,\vs_2)\varepsilon^3\left[(-\vh\cdot\nabla^3\vv^{B+}_{\vx})\vf^{B+}_{\vx}+(\vh\cdot\nabla^3\vv^{A+}_{\vx})\vf^{A+}_{\vx}\right]\right|\notag\\\le&\left|\partial_{20}V(\vp+\vs_1,\vs_2)\varepsilon^3|\vs_1|^3\left[(-\vh\cdot\nabla^3\vv^{B+}_{\vx})\vf^{B+}_{\vx}+(\vh\cdot\nabla^3\vv^{A+}_{\vx})\vf^{B+}_{\vx}\right]\right|\notag\\&+\left|\partial_{20}V(\vp+\vs_1,\vs_2)\varepsilon^3|\vs_1|^3\left[(-\vh\cdot\nabla^3\vv^{A+}_{\vx})\vf^{B+}_{\vx}+(\vh\cdot\nabla^3\vv^{A+}_{\vx})\vf^{A+}_{\vx}\right]\right|\notag\\\le&\left|\partial_{20}V(\vp+\vs_1,\vs_2)\varepsilon^4|\vs_1|^3\left[C\vv_{4,\vx,\vs_1}\vf^{B+}_{\vx}+(\vh\cdot\nabla^3\vv^{A+}_{\vx})D^{\vp}_0\vf^{A+}_{\vx}\right]\right|.\label{var3}
	\end{align}
	Similar property can be used for other $\varepsilon^3$ terms. Thus we notice $\varepsilon^3$ terms reduce to $\varepsilon^4$ terms in $R_\text{elas}$. As for the $O(\varepsilon^4)$ terms, we have
	\begin{align}
	&\lim_{R\to+\infty}\frac{1}{2R}\sum_{\langle\vs_1,\vs_2\rangle}\sum_{\vx\in\mathbb{L}_{R}} C  \varepsilon^4 V_{(t_1+1,t_2),s_1,s_2}|\vs_1|^{t_1}|\vs_2|^{t_2}\vv_{4,\vx,\vs_2}\vf_{\vx}^{A+}\notag\\\le&\varepsilon^2\sum_{\langle\vs_1,\vs_2\rangle}V_{(t_1,t_2),s_1,s_2}|\vs_1|^{t_1}|\vs_2|^{t_2+1}\|\vf\|_{X_\varepsilon}\notag\\\le&C\varepsilon^2\|\vf\|_{X_\varepsilon},\label{o4too2}
	\end{align}
	where the last two inequalities are from Lemma \ref{A4 come}, Lemma \ref{s come} and $t_1+t_2=3$. Thus the $R_\text{elas}$ is a $O(\varepsilon^2)$ term.
	
	2. \textbf{$R_\text{mis}$:}
	
	We divide the $R_\text{mis}$ into two parts:\begin{align}
	&R_\text{mis,1}\notag\\=&\lim_{R\to+\infty}\frac{\varepsilon^2}{2R}\sum_{\vx\in\mathbb{L}_{R}}\sum_{\vs}\Big[U'(\vs+\vp-\vd+\vv_{\vx+\vs}^{B+}-\vv_{\vx}^{A-})\notag\\&+U'(\vs+\vp-\vd+\vv_{\vx}^{B+}-\vv_{\vx-\vs}^{A-})-U'(\vs+\vp-\vd+\vv_{\vx}^{A+}-\vv_{\vx}^{A-})\notag\\&-U'(\vs+\vp-\vd+\vv_{\vx}^{B+}-\vv_{\vx}^{B-})\Big](\vf_{\vx}^{A+}-\vf_{\vx}^{B-})\notag\\&R_\text{mis,2}\notag\\=&\lim_{R\to+\infty}\frac{\varepsilon^2}{2R}\sum_{\vx\in\mathbb{L}_{R}}\sum_{\vs}\Big[U'(\vs+\vp-\vd+\vv_{\vx+\vs}^{B+}-\vv_{\vx}^{A-})(\vf_{\vx+\vs}^{B+}-\vf_{\vx}^{B+})\notag\\&+U'(\vs+\vp-\vd+\vv_{\vx}^{B+}-\vv_{\vx-\vs}^{A-})(\vf_{\vx}^{A-}-\vf_{\vx-\vs}^{A-})\notag\\&-U'(\vs+\vp-\vd+\vv_{\vx}^{A+}-\vv_{\vx}^{A-})(\vf_{\vx}^{A+}-\vf_{\vx}^{B+})\notag\\&-U'(\vs+\vp-\vd+\vv_{\vx}^{B+}-\vv_{\vx}^{B-})(\vf_{\vx}^{A-}-\vf_{\vx}^{B-})\Big]
	\end{align}
	
	As for the $R_\text{mis,1}$, we apply Taylor expansion at $U'(\vs+\vp-\vd+\vv_{\vx}^{B+}-\vv_{\vx}^{A-})$ and obtain \begin{align}
	&R_\text{mis,1}\notag\\\le&\lim_{R\to+\infty}\frac{\varepsilon^2}{2R}\sum_{\vx\in\mathbb{L}_{R}}\sum_{\vs}\Big[U'(\vs+\vp-\vd+\vv_{\vx}^{B+}-\vv_{\vx}^{A-})\notag\\&\times|\vv_{\vx+\vs}^{B+}+\vv_{\vx-\vs}^{A+}-\vv_{\vx}^{A+}-\vv_{\vx}^{B+}|(\vf_{\vx}^{A+}-\vf_{\vx}^{B-})\notag\\&+U_{3,\vs}\left(|\vv_{\vx+\vs}^{B+}-\vv_{\vx}^{B+}|^2+|\vv_{\vx}^{A+}-\vv_{\vx-\vs}^{A+}|^2+2|\vv_{\vx}^{A+}-\vv_{\vx}^{B+}|^2\right)\notag\\&\times(\vf_{\vx}^{A+}-\vf_{\vx}^{B-}).
	\end{align}
	Due to \begin{align}
	|\vv_{\vx+\vs}^{B+}+\vv_{\vx-\vs}^{A+}-\vv_{\vx}^{A+}-\vv_{\vx}^{B+}|&=\left|\varepsilon(\vs+\vp)\cdot\nabla \vv_{\vx}^{A+}-\varepsilon(\vs+\vp)\cdot\nabla \vv_{\vx}^{B+}+O(\varepsilon)\right|\notag\\&=\left|\varepsilon^2(\vs+\vp)\cdot\nabla^2 \vv_{\vx}^{A+}\cdot(-\vp)+O(\varepsilon^2)\right|\notag\\|\vv_{\vx+\vs}^{B+}-\vv_{\vx}^{B+}|^2&=\varepsilon^2|\vs\cdot\nabla \vv_{\vs}^{B+}+O(\varepsilon)|^2,
	\end{align}
	each term in $R_\text{mis,1}$ is the $\varepsilon^4$ terms. Similarly with $\ref{o4too2}$, we know $R_\text{mis,1}$ is a $\varepsilon^2$ term.
	
	For $R_\text{mis,2}$, due to the symmetric of $\vf$ and $\vv$, we obtain \begin{align}
	&R_\text{mis,2}=\lim_{R\to+\infty}\frac{\varepsilon^2}{2R}\sum_{\vx\in\mathbb{L}_{R}}\sum_{\vs}\Big[U'(\vs+\vp-\vd+\vv_{\vx+\vs}^{B+}-\vv_{\vx}^{A-})\notag\\&-U'(\vs+\vp-\vd+\vv_{\vx+2\vs}^{B+}-\vv_{\vx+\vs}^{A-})\Big](\vf_{2\vs+\vx}^{B+}-\vf_{\vx+\vs}^{A+})\notag\\&-\Big[U'(\vs+\vp-\vd+\vv_{\vx}^{A+}-\vv_{\vx}^{A-})-U'(\vs+\vp-\vd+\vv_{\vx}^{B+}-\vv_{\vx}^{B-})\Big]\notag\\&\times(\vf_{\vx}^{A+}-\vf_{\vx}^{B+})\notag\\=&\lim_{R\to+\infty}\frac{\varepsilon^2}{2R}\sum_{\vx\in\mathbb{L}_{R}}\sum_{\vs}U_{2,\vs}(-\varepsilon\nabla\vs\cdot\vv_{\vx+\vs}^{B+}-\varepsilon\vs\cdot\nabla\vv_{\vx}^{A+}+O(\varepsilon)(\varepsilon D_{\vs}^{\vp}\vf_{\vx+\vs}^{A+})\notag\\&-U_{2,\vs}(2\varepsilon\vp\cdot\nabla\vv_{\vx}^{A+}+O(\varepsilon))(\varepsilon D_{\vzero}^{\vp}\vf_{\vx+\vs}^{A+})\label{mis 2}
	\end{align}
	
	Therefore, each term in $R_\text{mis,2}$ is the $\varepsilon^4$ terms. Similarly with $\ref{o4too2}$, we know $R_\text{mis,1}$ is a $\varepsilon^2$ term.
	
\end{proof}

In the proof of the consistency, we assume the symmetric of $\vf$ and $\vv$. However, we can remove this assumption. For an example, in the $R_\text{mis,1}$, \begin{align}&\varepsilon^2\sum_{\vs}-U'(\vs+\vp-\vd+\vv_{\vx}^{A+}-\vv_{\vx}^{A-})(\vf_{\vx}^{A+}-\vf_{\vx}^{B+})\notag\\&-U'(\vs+\vp-\vd+\vv_{\vx}^{B+}-\vv_{\vx}^{B-})(\vf_{\vx}^{A-}-\vf_{\vx}^{B-})\notag\\=&\varepsilon^2\sum_{\vs}-\Big[U'(\vs+\vp-\vd+\vv_{\vx}^{A+}-\vv_{\vx}^{A-})-U'(\vs+\vp-\vd+\vv_{\vx}^{B+}-\vv_{\vx}^{B-})\Big]\notag\\&\times(\vf_{\vx}^{A+}-\vf_{\vx}^{B+})-U'(\vs+\vp-\vd+\vv_{\vx}^{B+}-\vv_{\vx}^{B-})(\vf_{\vx}^{A-}-\vf_{\vx}^{B-}+\vf_{\vx}^{A+}-\vf_{\vx}^{B+})\notag\end{align}

We already know the first term is the second order terms due to Eqs.~(\ref{mis 2}). The second term is a new term since we do not assume the symmetric of $\vf$. However, we denote $\vf_{\vx}^{A+}+\vf_{\vx}^{A-}=\varepsilon\bar{D}_0\vf_{\vx}^{A+}$. Hence, we obtain \begin{align}
&\varepsilon^2\sum_{\vs}-U'(\vs+\vp-\vd+\vv_{\vx}^{B+}-\vv_{\vx}^{B-})(\vf_{\vx}^{A-}-\vf_{\vx}^{B-}+\vf_{\vx}^{A+}-\vf_{\vx}^{B+})\notag\\=&\varepsilon^2\sum_{\vs}-U'(\vs+\vp-\vd+\vv_{\vx}^{B+}-\vv_{\vx}^{B-})(\varepsilon\bar{D}_0\vf_{\vx}^{A+}-\varepsilon\bar{D}_0\vf_{\vx}^{B+}).\notag
\end{align}

Similarlity, we can obtain $$\varepsilon^2\sum_{\vs}-U'(\vs+\vp-\vd+\vv_{\vx}^{A+}-\vv_{\vx}^{A-})(\varepsilon\bar{D}_0\vf_{\vx}^{B+}-\varepsilon\bar{D}_0\vf_{\vx}^{A+})$$ from $R_\text{mis,2}$. Then we combining them and find each term is still $\varepsilon^4$ order.


\section{Stability of Atomistic Model and Proof of Theorem 2}\label{sec:stability}

In this section, we prove stability of the atomistic model for the complex lattice. Since we already have stability of the Peierls--Nabarro model, our method is to show that the gap between $\delta^2E_\text{PN}[v]$ and $\delta^2E_\text{a}[v]$ is small. Here we are not able to prove the stability directly as the consistency because we do not know the exact value of $\delta^2E_\text{PN}[v]$ in the discrete space. More precisely, in the consistency of the Peierls--Nabarro model, we know $\delta E_\text{PN}[v]=0$ because $v$ is the solution of the Peierls--Nabarro model; whereas for $\delta^2E_\text{PN}[v]$, we just know it is positive (stability of Peierls--Nabarro model Proposition \ref{stability of PN}) in the continuum space.

Therefore, we construct interpolation polynomials on the complex lattice in order to estimate the energies of the two models and to compare two models. However, there are four kinds of intralayer interactions (elastic energy) and four kinds of inter-layer interactions (misfit energy). It is difficult to describe the complex lattice by a single interpolation polynomial. Even if we establish a single interpolation polynomials for complex lattice, the order of the polynomial will be very high and the resulting energy will have unphysical oscillations. In order to solve this problem, we construct two interpolation polynomials to describe the energy on the complex lattice, and each polynomial  describes a part of the interactions in the atomistic model.

\subsection{Atomistic Dislocation Condition (ADC)}\label{sec:adc}
In this subsection, we introduce the Atomistic Dislocation Condition (ADC).

We focus on the upper layer. Consider an A atom  located on $\{x=0\}$ as shown in Figure \ref{value}. By assumption, the $x$-component of the displacement of this atom is $\frac{1}{4}$.
If on the crystallographic line where atoms  have the same $y$ coordinate, there is no atom at $x=0$, e.g., atoms B and C in Figure \ref{value},  we  assume that the average of the displacement about the two atoms is $\frac{1}{4}$. That is,
\begin{equation}\label{eqn:adc}
u_A=\frac{1}{4}, \ \ {\rm or}\ \  \frac{u_C+u_B}{2}=\frac{1}{4},
\end{equation}
for all atoms on or near $x=0$. More precisely, if there is an atom located at $x=0$ on a crystallographic line with the same $y$ coordinate, the first equation in \eqref{eqn:adc} holds; while if there no atom located at $x=0$ on a crystallographic line with the same $y$ coordinate,  the second equation in \eqref{eqn:adc} holds; see Figure \ref{value}.
 If the atomistic displacement field on the discrete space satisfies this assumption, we say that the atomistic model satisfies \textbf{Atomistic Dislocation Condition (ADC)}.
\begin{figure}[h]
	\centering
	\scalebox{0.30}{\includegraphics{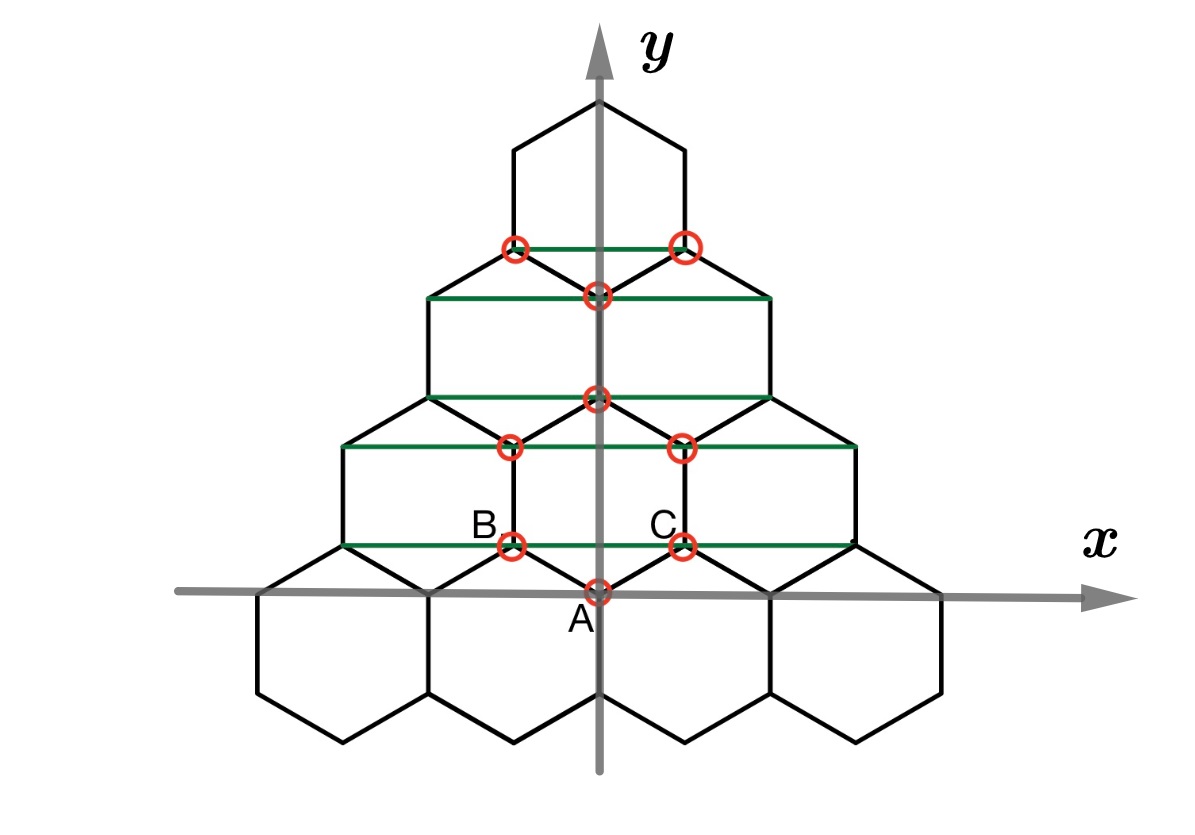}}
	\caption{Atomistic Dislocation Condition (ADC).}
	\label{value}
\end{figure}


\subsection{Interpolation Polynomials of Complex Lattice}
In this subsection, we construct interpolation polynomials for estimating the energy  of the atomistic model with complex lattice, which enables us to prove the  stability of atomistic model  by comparing the atomistic model and the PN model directly. 

Now we construct the interpolation polynomials for a test function $\vf\in X_\varepsilon$. Two different interpolation polynomials are constructed on the complex lattice based on the two kind of atoms, A and B, and the four kinds of intralayer interactions, AAA, AAB and BBB, BBA. Here we focus on the intralayer interactions because the interlayer interactions can be handled in a relatively easier way.

\begin{figure}[h]
	\centering
\centering
	\subfigure[]{\includegraphics[width=0.65\textwidth]{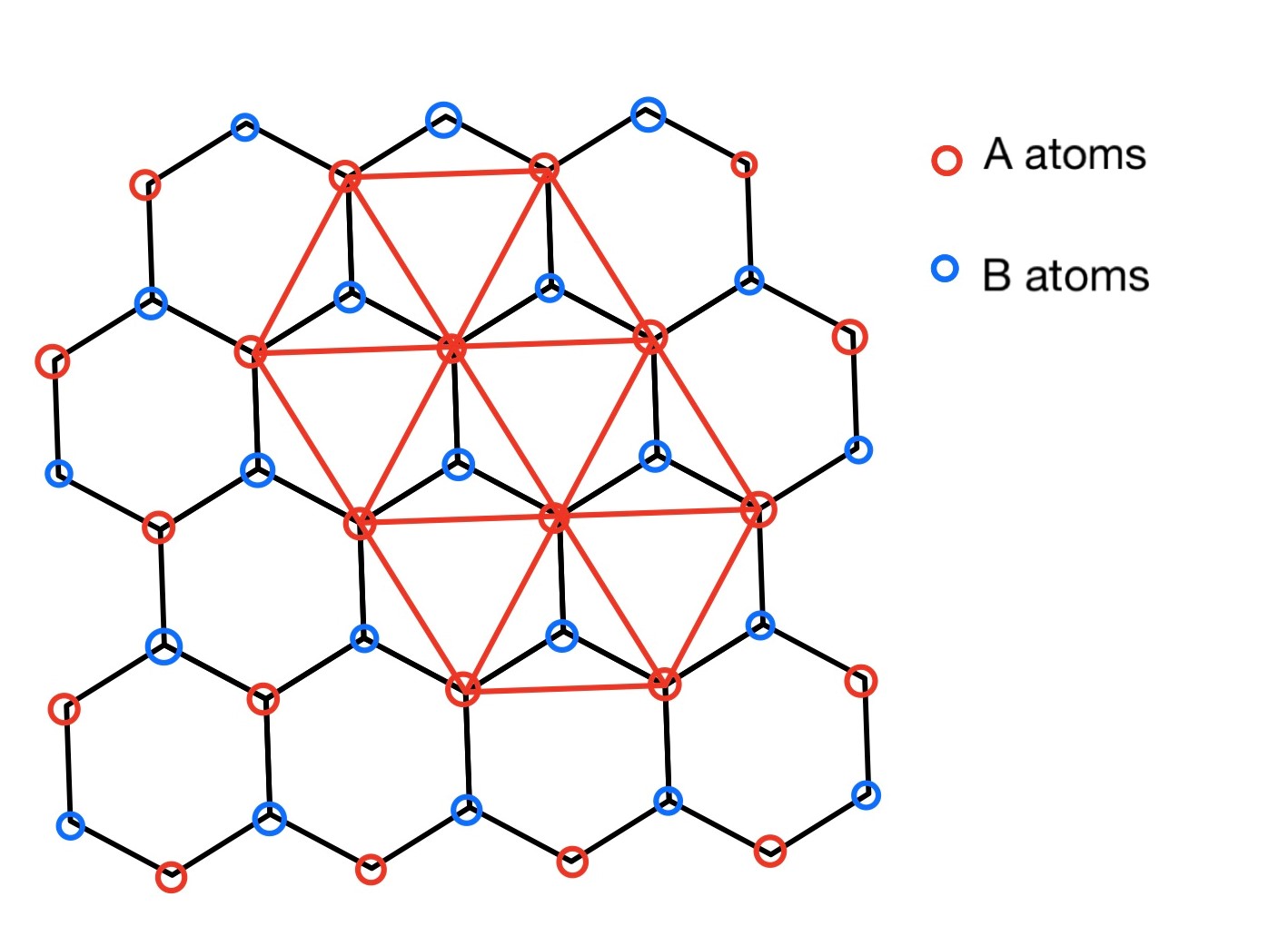}}
	\subfigure[]{\includegraphics[width=0.65\textwidth]{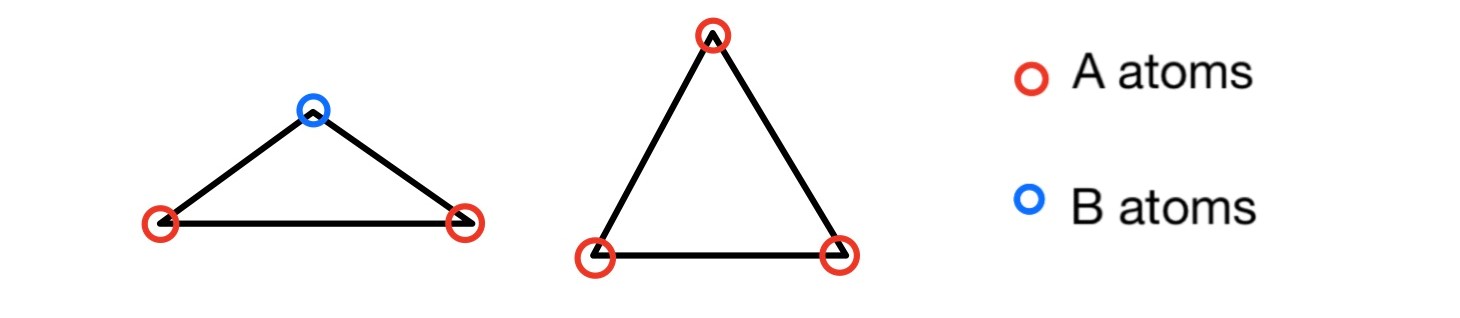}}
	\caption{(a) AAA and AAB interactions using a single interpolation polynomial. (b) Two kinds of triangles.}
	\label{Interpolation Polynomial for A}
\end{figure}

We first consider the AAA and AAB interactions. By connecting the neighboring A atoms,   $\mathbb{R}^2$ plane is divided into triangles; see Figure~\ref{Interpolation Polynomial for A}(a). As shown in Figure~\ref{Interpolation Polynomial for A}(b), 
 there are only two kind of triangles, whose three vertices are A atoms (AAA interactions) and  two A atoms and one B atom (AAB interactions), respectively.

For each triangle, we construct the linear Lagrange interpolation for the interaction energy.  Thus we get a piecewise linear interpolating function in $\mathbb{R}^2$. Below we give a brief introduction to linear Lagrange interpolation on a triangle.  More details can be found, e.g. in \cite{brenner2007mathematical}.

Suppose $T(x,y)=Ax+By+C$ is an interpolation polynomial in each equilateral triangles. We denote three vertices in the equilateral triangle as $(x_1,y_1)$, $(x_2,y_2)$ and $(x_3,y_3)$, and the value of $f$ in these three vertices as $T_1$, $T_2$ and $T_3$. By these notations, $$
\left[\begin{array}{l}{T_{1}} \\ {T_{2}} \\ {T_{3}}\end{array}\right]=\left[\begin{array}{lll}{1} & {x_{1}} & {y_{1}} \\ {1} & {x_{2}} & {y_{2}} \\ {1} & {x_{3}} & {y_{3}}\end{array}\right]\left[\begin{array}{l}{C} \\ {A} \\ {B}\end{array}\right].
$$
Solving this linear system, we obtain
\begin{equation}
\begin{array}{l}{A=\left[\left(y_{3}-y_{1}\right) \cdot\left(T_{2}-T_{1}\right)-\left(y_{2}-y_{1}\right) \cdot\left(T_{3}-T_{1}\right)\right] / \Delta}, \\ {B=\left[\left(x_{2}-x_{1}\right) \cdot\left(T_{3}-T_{1}\right)-\left(x_{3}-x_{1}\right) \cdot\left(T_{2}-T_{1}\right)\right] / \Delta}, \\ {\Delta=\left(x_{2}-x_{1}\right) \cdot\left(y_{3}-y_{1}\right)-\left(x_{3}-x_{1}\right) \cdot\left(y_{2}-y_{1}\right)}, \end{array}\label{triangles}
\end{equation}
Here $\Delta$ is twice the area of the triangle.

Now we prove that the piecewise linear function belongs to $X_0$.
\begin{mylem}\label{well define}
	Suppose $\vf=\left\{\vf^+,\vf^-\right\}\in X_\varepsilon$. We can find a piecewise linear function $\vf_A=\left\{\vf_A^+,\vf_A^-\right\}\in X_0$, such that $\vf=\vf_A$ on $\mathbb{L}_{\vp}$.
\end{mylem}
\begin{proof}
	Without loss of generality, we consider  $f^+\in X_\varepsilon$, where $\vf^+=(f^+,0)$. Since $\mathbb{R}^2$ is divided into triangles as shown in Figure~\ref{Interpolation Polynomial for A}(a) and each triangles has a interpolation polynomial, we define $f_A^+$ in $\mathbb{R}^2$ by combining all the interpolation polynomials, and  $f_A^+$ is a linear function on each triangle. Since the value of $f_A^+$ on boundary of triangles come from linear combination of the values at two vertices, we know that the two interpolation polynomials on the two neighboring triangles have the same value on the boundary where they meet. Thus $f_A^+$ is well-defined.
Moreover, since $f_A^+$ is continuous in each triangles, $f_A^+\in C(\mathbb{R}^2)$.

	
	Finally, we prove that  $f_A^+$ satisfies the boundary conditions of $X_0$. Since $f_A^++\frac{1}{4}$ defined on $\mathbb{L}_{\vp}$ satisfies ADC in Eq.~\eqref{eqn:adc}, we know that for a triangle that containing the $y$-axis, e.g., the triangle ABC in Figure~\ref{value}, we have 
$$f_A^+|_A=0, \ \ \frac{f_A^+|_B+f_A^+|_C}{2}=0.$$
 By direct calculation like \eqref{triangles}, we have $f_A^+(0,y)=0$ inside this triangle. Therefore,$$f_A^+(0,y)=0\text{ for all } y\in \mathbb{R}.$$
\end{proof}

Similarly, another linear interpolating function can be constructed based on the  ABB and BBB interactions. In this case, the $\mathbb{R}^2$ plane is into triangles by connecting neighboring B atoms.


\subsection{Stability of Atomistic Model}
In this part, we prove the stability of atomistic model, which means we are going to prove
\begin{equation}
\left\langle\delta^{2} E_\text{a}[\vv] \vf, \vf\right\rangle_{\varepsilon}\ge C\|\vf\|^2_{X_\varepsilon},\label{goal}
\end{equation}
for $C>0$ and $\vf\in X_\varepsilon$. In addition, we get that \eqref{goal} is still vaild when $\vu$ is close to $\vv$, where $\vv$ is the solution of PN model. The method that we use is energy estimate by two interpolation polynomials defined in Section 9.2.

First we prove the stability gap between $E_\text{a}^A[\vv]$ and $E_{\mathrm{PN}}^A[\vv]$ based on the interpolation polynomials.

\begin{mylem}\label{change}
	Suppose that A1--A6 hold, and $\vv$ is the dislocation solution of the PN model in Theorem \ref{Theorem 1}. Then there exists  constants $\varepsilon_0$ and $C$, such that for $0<\varepsilon<\varepsilon_0$ and $\vf\in X_{\varepsilon}$, we have
	\begin{align}
	&\left\langle\delta^{2} E_\text{a}^A[\vv] \vf, \vf\right\rangle_{\varepsilon}-\left\langle\delta^{2} E_{\mathrm{PN}}^A[\vv] {\vf}_A, {\vf}_A\right\rangle_{0}\notag\\=&\left\langle\delta^{2} E_\text{a}^A[\vzero] \vf, \vf\right\rangle_{\varepsilon}-\left\langle\delta^{2} E_{\mathrm{PN}}^A[\vzero] {\vf}_A, {\vf}_A\right\rangle_{0}+O(\varepsilon)\|\vf\|_{X_{\varepsilon}}^{2}.\label{lem 10}
	\end{align}
\end{mylem}
\begin{proof}
	By direct calculation, we have
	\begin{flalign}
	&\left\langle\delta^{2} E_\text{a}^A[\vv] \vf, \vf\right\rangle_{\varepsilon}\notag\\=&\lim_{R\to+\infty}\frac{1}{2R}\sum_{\vx\in\mathbb{L}_{R}}\Big\{\frac{\varepsilon^2}{6}\sum_{\langle\vs_1,\vs_2\rangle}\sum_{i+j=2}\tbinom{2}{i}\partial_{ij}V(\vs_1+\varepsilon D_{\vs_1}\vv_{\vx}^{A\pm},\vs_2+\varepsilon D_{\vs_2}\vv_{\vx}^{A\pm})\notag\\&(D_{\vs_1}\vf_{\vx}^{A\pm})^i(D_{\vs_2}\vf_{\vx}^{A\pm})^j\notag\\&+\frac{\varepsilon^2}{2}\sum_{\langle\vs_1,\vs_2\rangle}\sum_{i+j=2}\tbinom{2}{i}\partial_{ij}V(\vp+\vs_1+\varepsilon D^{\vp}_{\vs_1}v_{\vx}^{A\pm},\vs_2+\varepsilon D_{\vs_2}v_{\vx}^{A\pm})\notag\\&(D^{\vp}_{\vs_1}\vf_{\vx}^{A\pm})^i(D_{\vs_2}\vf_{\vx}^{A\pm})^j\notag\\&+\varepsilon^2\sum_{\vs}\big[U''({\vv^{A+}_{\vx+\vs}-\vv^{A-}_{\vx}}+\vs-\vd)(\vf^{A+}_{\vx+\vs}-\vf^{A-}_{\vx})^2\notag\\&+U''({\vv^{A+}_{\vx+\vs}-\vv^{B-}_{\vx}}+\vs-{\vp}-\vd)(\vf^{A+}_{\vx+\vs}-\vf^{B-}_{\vx})^2\big]\Big\}\label{A have part},\\&\left\langle\delta^{2} E_\text{a}^A[\vzero] \vf, \vf\right\rangle_{\varepsilon}\notag\\=&\lim_{R\to+\infty}\frac{1}{2R}\sum_{\vx\in\mathbb{L}_{R}}\Big\{\frac{\varepsilon^2}{6}\sum_{\langle\vs_1,\vs_2\rangle}\sum_{i+j=2}\tbinom{2}{i}\partial_{ij}V(\vs_1,\vs_2)(D_{\vs_1}\vf_{\vx}^{A\pm})^i(D_{\vs_2}\vf_{\vx}^{A\pm})^j\notag\\&+\frac{\varepsilon^2}{2}\sum_{\langle\vs_1,\vs_2\rangle}\sum_{i+j=2}\tbinom{2}{i}\partial_{ij}V(\vp+\vs_1,\vs_2)(D^{\vp}_{\vs_1}\vf_{\vx}^{A\pm})^i(D_{\vs_2}\vf_{\vx}^{A\pm})^j\notag\\&+\varepsilon^2\sum_{\vs}\big[U''(\vs-\vd)(\vf^{A+}_{\vx+\vs}-\vf^{A-}_{\vx})^2+U''(\vs-{\vp}-\vd)(\vf^{A+}_{\vx+\vs}-\vf^{B-}_{\vx})^2\big]\Big\}\label{A no part}\\&\left\langle\delta^{2} E_{\mathrm{PN}}^A[\vv] {\vf}_A, {\vf}_A\right\rangle_{0}\notag\\=&\lim_{R\to+\infty}\frac{1}{R}\int_{-R}^{R}\int_{\mathbb{R}}\left[\alpha_1(\partial_y\vf^+_A)^2+\alpha_2(\partial_y\vf^+_A)^2+2\nabla^2\gamma_A(\vphi)(\vf_A^\perp)^2\right]\,\D x\,\D y,\label{ P have part}\\
	&\left\langle\delta^{2} E_{\mathrm{PN}}^A[\vzero] {\vf}_A, {\vf}_A\right\rangle_{0}\notag\\=&\lim_{R\to+\infty}\frac{1}{R}\int_{-R}^{R}\int_{\mathbb{R}}\left[\alpha_1(\partial_y\vf^+_A)^2+\alpha_2(\partial_y\vf^+_A)^2+2\nabla^2\gamma_A(\vzero)(\vf_A^\perp)^2\right]\,\D x\,\D y.\label{ P no part}
	\end{flalign}
	We divide \eqref{lem 10} into following five parts, $R_i$, $i=1,2,\cdots,5$. The remaining problem is to show that these five parts are $O(\varepsilon)$.
	
	1. \begin{align}
	R_1=&\varepsilon^2\lim_{R\to+\infty}\frac{1}{2R}\sum_{\vx\in\mathbb{L}_{R}}\Big\{\frac{1}{6}\sum_{\langle\vs_1,\vs_2\rangle}\sum_{i+j=2}\tbinom{2}{i}\notag\\&\partial_{ij}V(\vs_1+\varepsilon D_{\vs_1}\vv_{\vx}^{A\pm},\vs_2+\varepsilon D_{\vs_2}\vv_{\vx}^{A\pm})(D_{\vs_1}\vf_{\vx}^{A\pm})^i(D_{\vs_2}\vf_{\vx}^{A\pm})^j\notag\\+&\frac{1}{2}\sum_{\langle\vs_1,\vs_2\rangle}\sum_{i+j=2}\tbinom{2}{i}\partial_{ij}V(\vp+\vs_1+\varepsilon D^{\vp}_{\vs_1}\vv_{\vx}^{A\pm},s_2+\varepsilon D_{\vs_2}\vv_{\vx}^{A\pm})\notag\\&(D^{\vp}_{\vs_1}\vf_{\vx}^{A\pm})^i(D_{\vs_2}\vf_{\vx}^{A\pm})^j\notag\\-&\frac{1}{6}\sum_{\langle\vs_1,\vs_2\rangle}\sum_{i+j=2}\tbinom{2}{i}\partial_{ij}V(\vs_1,\vs_2)(D_{\vs_1}\vf_{\vx}^{A\pm})^i(D_{\vs_2}\vf_{\vx}^{A\pm})^j\notag\\-&\frac{1}{2}\sum_{\langle\vs_1,\vs_2\rangle}\sum_{i+j=2}\tbinom{2}{i}\partial_{ij}V(\vp+\vs_1,\vs_2)(D^{\vp}_{\vs_1}\vf_{\vx}^{A\pm})^i(D_{\vs_2}\vf_{\vx}^{A\pm})^j\Big\}.\label{R1}
	\end{align}
	Without loss of generality, we just calculate the complex lattice parts with $i=2,j=0$:
	\begin{align}
	&\lim_{R\to+\infty}\frac{1}{2R}\sum_{x\in\mathbb{L}_{R}} \sum_{\langle\vs_1,\vs_2\rangle}[\varepsilon^2\partial_{20}V(\vp+\vs_1+\varepsilon D^{\vp}_{\vs_1}\vv_{\vx}^{A\pm},\vs_2+\varepsilon D_{\vs_2}\vv_{\vx}^{A\pm})\notag\\&-\partial_{20}V(\vp+\vs_1,\vs_2)](D^{\vp}_{\vs_1}\vf_{\vx}^{A\pm})^2\notag\\\le&\varepsilon^3\lim_{R\to+\infty}\frac{1}{2R}\sum_{\vx\in\mathbb{L}_{R}} \sum_{\langle\vs_1,\vs_2\rangle}(V_{(3,0),\vp+\vs_1,\vs_2}+V_{(2,1),\vp+\vs_1,\vs_2})\notag\\&(\vv^{\vp}_{1,\vx,\vs_1}+\vv_{1,\vx,\vs_1})(D^{\vp}_{\vs_1}\vf_{\vx}^{A\pm})^2\notag\\\le&C\varepsilon \|\vf\|_\varepsilon^2,
	\end{align}
	where the reason is the same as \eqref{o4too2}.
	
	2.\begin{align}
	R_2=&\lim_{R\to+\infty}\frac{\varepsilon^2}{2R}\sum_{\vx\in\mathbb{L}_{R}} \sum_{\vs}\big[U''({\vv^{A+}_{\vx+\vs}-\vv^{A-}_{\vx}}+\vs-\vd)(\vf^{A+}_{\vx+\vs}-\vf^{A-}_{\vx})^2\notag\\&+U''({\vv^{B+}_{\vs+\vx}-\vv^{A-}_{\vx}}+\vs-{\vp}-\vd)(\vf^{B+}_{\vx+\vs}-\vf^{A-}_{\vx})^2\notag\\&-U''({\vv^{A+}_{\vx}-\vv^{A-}_{\vx}}+\vs-\vd)(\vf^{A+}_{\vx+\vs}-\vf^{A-}_{\vx})^2\notag\\&-U''({\vv^{A+}_{\vx}-\vv^{A-}_{\vx}}+\vs-{\vp}-\vd)(\vf^{B+}_{\vx+\vs}-\vf^{A-}_{\vx})^2\big].
	\end{align}
	Without loss of generality, we just calculate the complex lattice parts:
	\begin{align}
	&\lim_{R\to+\infty}\frac{\varepsilon^2}{2R}\sum_{x\in\mathbb{L}_{R}} \sum_{\vs}\big[U''({\vv^{B+}_{\vs+\vx}-\vv^{A-}_{\vx}}+\vs+{\vp})\notag\\&-U''({\vv^{A+}_{\vx}-\vv^{A-}_{\vx}}+\vs+{\vp})\big](\vf^{B+}_{\vx+\vs}-\vf^{A-}_{\vx})^2\notag\\\le&\varepsilon\lim_{R\to+\infty}\frac{1}{2R}\sum_{\vx\in\mathbb{L}_{R}} \varepsilon^2\sum_{\vs}U_{3,\vs+\vp}(\varepsilon D_{\vs}\vf_{\vx}^{B+}+\vf_{\vx}^{B+}+\vf_{\vx}^{A+})^2\notag\\\le&C\varepsilon \|\vf\|_\varepsilon^2,
	\end{align}
	where the reason is the same as \eqref{o4too2}.
	
	3.\begin{align}
	R_3=&\lim_{R\to+\infty}\frac{\varepsilon^2}{2R}\sum_{\vx\in\mathbb{L}_{R}} \sum_{\vs}\Big\{\big[U''({\vv^{A+}_{\vx}-\vv^{A-}_{\vx}}+\vs-\vd) -U''(\vs-\vd)\big]\notag\\&\big[(\vf^{A+}_{\vx+\vs}-\vf^{A-}_{\vx})^2-(\vf^{A+}_{\vx}-\vf^{A-}_{\vx})^2\big]\notag\\&+\big[U''({\vv^{A+}_{\vx}-\vv^{A-}_{\vx}}+\vs+\vp-\vd) -U''(\vs+\vp-\vd)\big]\notag\\&\big[(\vf^{B+}_{\vx+\vs}-\vf^{A-}_{\vx})^2-(\vf^{A+}_{\vx}-\vf^{A-}_{\vx})^2\big]\Big\}.
	\end{align}
	By the same way in $R_2$, we have $$R_3\le C\varepsilon \|\vf\|_\varepsilon^2.$$

\begin{figure}[htbp]
		\centering
		\scalebox{0.35}{\includegraphics{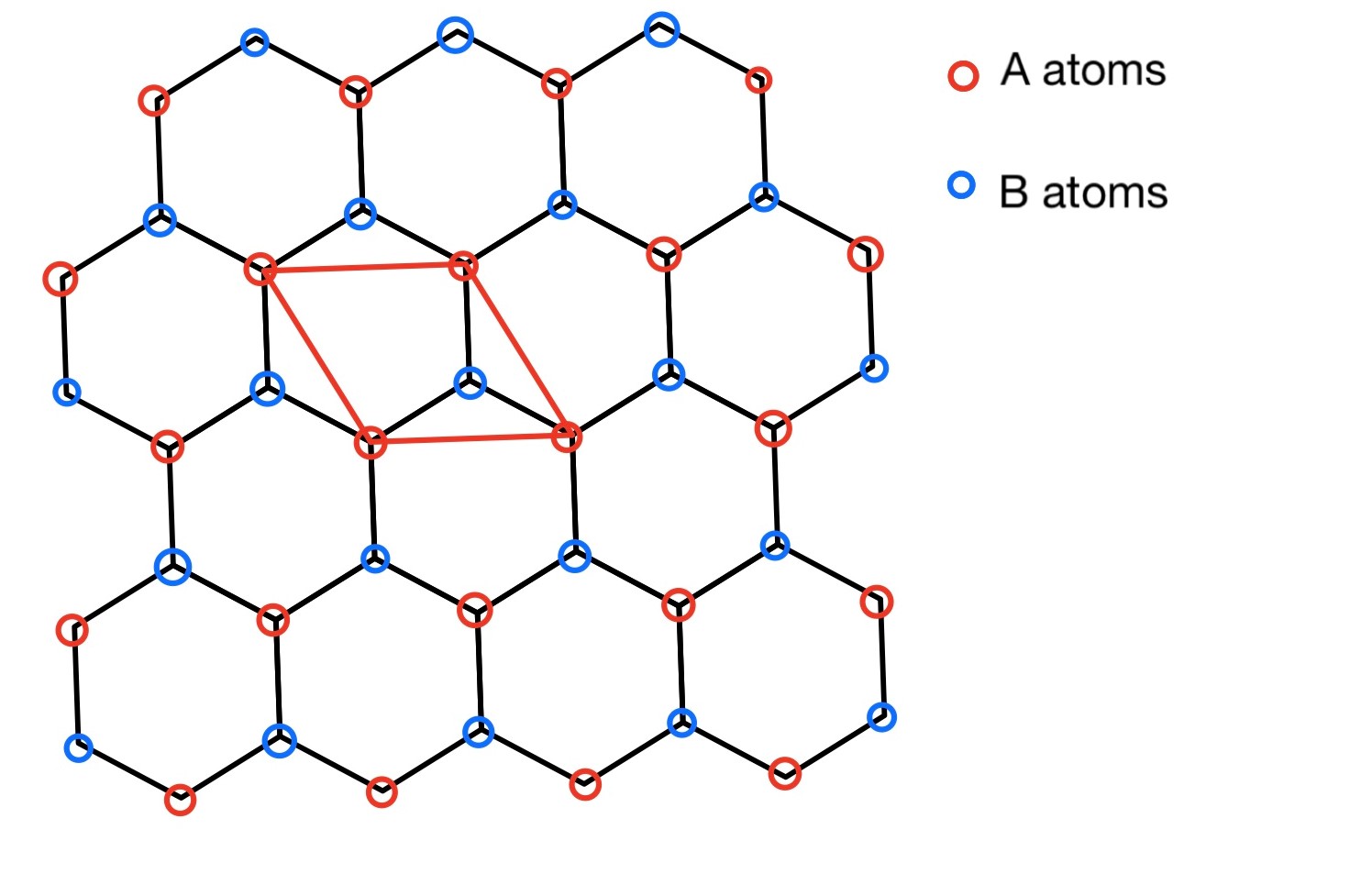}}
		\caption{Tessellation of $\mathbb{R}^2$.}
		\label{divide 2} 
	\end{figure}
	
	4.	The next two terms, $R_4$ and $R_5$, will be compared with the continuum model.
		We divide the $\mathbb{R}^2$ by parallelograms as shown in Figure~\ref{divide 2}. For each parallelogram, there is a B atom. When the label of B atom is $\vs$, we denote this  parallelogram as $P_{\vs}$.
	\begin{align}
	R_4=&\sum_{\vs}\sum_{\vx}\lim_{R\to+\infty}\frac{1}{2R}\int_{P_{\vx}\cap (-R,R)\times \mathbb{R}}\sum_{\vs}\Big\{\big[U''(\vs+\vv_{\vx}^{A+}-\vv_{\vx}^{A-}-\vd)\notag\\&-U''(\vs+\vv^+-\vv^--\vd)\big](\vf^{A+}_{\vx}-\vf^{A-}_{\vx})^2+\big[U''(\vs+\vp+\vv_{\vx}^{A+}-\vv_{\vx}^{A-}-\vd)\notag\\&-U''(\vs+\vp+\vv^+-\vv^--\vd)\big](\vf^{A+}_{\vx}-\vf^{A-}_{\vx})^2\Big\}\,\D x\,\D y.
	\end{align}
	Since $$\Big|U''(\vs+\vp+\vv_{\vx}^{A+}-\vv_{\vx}^{A-}-\vd)-U''(\vs+\vp+\vv^+-\vv^--\vd)\Big|\le \varepsilon U_{3,\vs+\vp}\|\nabla \vv\|_\infty,$$ we can get
	\begin{equation}
	R_4\le C\varepsilon \|\vf\|_\varepsilon^2.\label{R4}
	\end{equation}

	5.\begin{align}
	R_5=&\sum_{\vs}\sum_{\vx}\lim_{R\to+\infty}\frac{1}{2R}\int_{P_{\vx}\cap (-R,R)\times \mathbb{R}}\sum_{\vs}\Big\{\big[U''(\vs+\vv^+-\vv^-)\notag\\&-U''(\vs)\big][(\vf^{A+}_{\vx}-\vf^{A-}_{\vx})^2-(\vf_A^+-\vf_A^-)^2]\notag\\& + \big[U''(\vs+\vp+\vv^+-\vv^-)\notag\\&-U''(\vs+\vp)\big][(\vf^{A+}_{\vx}-\vf^{A-}_{\vx})^2-(\vf_A^+-\vf_A^-)^2]\Big\}\,\D x\,\D y .
	\end{align}
	Since we have $\vf=\{(f^+,0),(f^-,0)\}$, \begin{align}
	&(f^{A+}_{\vx}-f^{A-}_{\vx})^2-(f_A^+-f_A^-)^2\notag\\=&(  f^{A+}_{\vx}+f_A^+-f^{A-}_{\vx}-f_A^-)(f^{A+}_{\vx}-f_A^++f^{A-}_{\vx}+f_A^-)\notag
	\end{align}
	In addition, \begin{align}&\int_{P_{\vx}}2f_A^+\,\D x\,\D y\\=&\frac{\sqrt{3}}{2}\varepsilon^2\left[\frac{1}{3}(f^{A+}_{\vx+\ve_2-\ve_1}+f^{B+}_{\vx})+\frac{5}{9}(f^{A+}_{\vx}+f^{A+}_{\vx+\ve_2})+\frac{2}{9}f^{A+}_{\vx+\ve_1}\right].\end{align}
	Therefore we can have
	\begin{align}
	&\int_{P_{x}}\big[(f^{A+}_{\vx}-f^{A-}_{\vx})^2-(f_A^+-f_A^-)^2\big]\,\D x\,\D y\notag\\\le&\frac{\sqrt{3}}{2}\varepsilon^32(f_{\vx}^{A+}+f^{A+}_{\vx+\ve_2-\ve_1}+f^{B+}_{\vx}+f^{A+}_{\vx+\ve_2}+f^{A+}_{\vx+\ve_1})\notag\\&\times(D_{\ve_1+\ve_2}f^{A+}_{\vx}+D_{\ve_2}f^{A+}_{\vx}+D_{\ve_1}f^{A+}_{\vx}+D^{\vp}_{\vzero}f^{A+}_{\vx}).
	\end{align}
	Thus \begin{align}R_5\le&  C\varepsilon^3 \sum_{\vs}\sum_{\vx}U_{3,\vs+\vp}2(\vf_{\vx}^{A+}+\vf^{A+}_{\vx+\ve_2-\ve_1}+\vf^{B+}_{\vx}+\vf^{A+}_{\vx+\ve_2}+\vf^{A+}_{\vx+\ve_1})\notag\\&\times(D_{\ve_1+\ve_2}\vf^{A+}_{\vx}+D_{\ve_2}\vf^{A+}_{\vx}+D_{\ve_1}\vf^{A+}_{\vx}+D^{\vp}_{\vzero}\vf^{A+}_{\vx})\notag\\ \le& C\varepsilon\|\vf\|_\varepsilon^2.\label{R5}
	\end{align}
\end{proof}

Next we define the stability gap to prove the stability of atomistic model.

\begin{mylem}\label{Delta}
	Suppose that A1--A6 hold, and $\vv$ is the dislocation solution of the PN model in Theorem \ref{Theorem 1}. There exist constants $\varepsilon_0$ and $C$, such that for $0<\varepsilon<\varepsilon_0$ and $\vf\in X_{\varepsilon}$, we have
	$$\left\langle\delta^{2} E^A_\text{a}[\vzero] \vf, \vf\right\rangle_{\varepsilon}-\left\langle\delta^{2} E^A_\text{PN}[\vzero] {\vf_A}, {\vf_A}\right\rangle_{0} \geq-\Delta_A\|\vf\|_{X_{\varepsilon}}^{2}+O(\varepsilon)\|\vf\|_{X_{\varepsilon}}^{2},$$
	where $\Delta_A$ we show in the proof.
\end{mylem}
\begin{proof}
	Thanks to \eqref{A no part} and \eqref{ P no part}, we only need to calculate the following parts:
	\begin{align}
	&R_\text{g1}\notag\\=&\lim_{R\to+\infty}\frac{1}{2R}\sum_{\vx\in\mathbb{L}_{R}}\Big\{\frac{\varepsilon^2}{6}\sum_{\langle\vs_1,\vs_2\rangle}\sum_{i+j=2}\tbinom{2}{i}\partial_{ij}V(\vs_1,\vs_2)(D_{\vs_1}\vf_{\vx}^{A\pm})^i(D_{\vs_2}v_{\vx}^{A\pm})^j\notag\\&+\frac{\varepsilon^2}{2}\sum_{\langle\vs_1,\vs_2\rangle}\sum_{i+j=2}\tbinom{2}{i}\partial_{ij}V(\vp+\vs_1,\vs_2)(D^{\vp}_{\vs_1}\vf_{\vx}^{A\pm})^i(D_{\vs_2}\vf_{\vx}^{A\pm})^j\notag\\&-\lim_{R\to+\infty}\frac{1}{R}\int_{-R}^{R}\int_{\mathbb{R}}\left[\alpha_1(f_A)_x^2+\alpha_2(f_B)_y^2\right]\,\D x\,\D y.\label{g1}\end{align}
	\begin{align}
	R_\text{g2}=&\lim_{R\to+\infty}\frac{1}{2R}\sum_{\vx\in\mathbb{L}_{R}}\Big\{\varepsilon^2\sum_{\vs}\big[U''(\vs-\vd)(\vf^{A+}_{\vx+\vs}-\vf^{A-}_{\vx})^2+U''(\vs-{\vp}-\vd)\notag\\&\times(\vf^{B+}_{\vx+\vs}-\vf^{A-}_{\vx})^2\big]\Big\}-\lim_{R\to+\infty}\frac{1}{R}\int_{-R}^{R}\int_{\mathbb{R}}2\nabla^2\gamma(\vzero)\vf^2_A\,\D x\,\D y.\label{g2}
	\end{align}
	Due to \eqref{R4} and \eqref{R5}, we have $$R_\text{g2}\le C\varepsilon \|\vf\|_{X_{\varepsilon}}^{2}.$$
	By the the definition of $\vf_A$ and \eqref{triangles}, we have
	\begin{align}
	R_\text{g1}=&\lim_{R\to+\infty}\frac{1}{2R}\sum_{\vx\in\mathbb{L}_{R}}\sum_{\langle\vs_1,\vs_2\rangle\backslash{A_1}}\sum_{i+j=2}\tbinom{2}{i}\notag\\&\Bigg[\frac{\varepsilon^2}{6}\partial_{ij}V(\vs_1,\vs_2)(D_{\vs_1}\vf_{\vx}^{A\pm})^i(D_{\vs_2}\vv_{\vx}^{A\pm})^j\notag\\&+\frac{\varepsilon^2}{2}\sum_{\langle\vs_1,\vs_2\rangle\backslash{A_2}}\sum_{i+j=2}\tbinom{2}{i}\partial_{ij}V(\vp+\vs_1,\vs_2)(D^{\vp}_{\vs_1}\vf_{\vx}^{A\pm})^i(D_{\vs_2}\vf_{\vx}^{A\pm})^j\Bigg]\notag\\&-\lim_{R\to+\infty}\frac{1}{4\sqrt{3}R}\int_{-R}^{R}\int_{\mathbb{R}}\sum_{i+j=2}\tbinom{2}{i}\notag\\&\Bigg\{\frac{1}{3}\sum_{\langle\vs_1,\vs_2\rangle\backslash{A_1}}\partial_{ij}V(\vs_1,\vs_2)
	\big[(\vs_1)\cdot(\nabla {\vf_A}^\pm)\big]^i\big[(\vs_2)\cdot(\nabla {\vf_A}^\pm)\big]^j-\sum_{\langle\vs_1,\vs_2\rangle\backslash{A_2}}\notag\\&\partial_{ij}V(\vs_1+ \vp,\vs_2)\big[(\vs_1+\vp)\cdot(\nabla {\vf_A}^\pm)\big]^i\big[(\vs_2)\cdot(\nabla {\vf_A}^\pm)\big]^j\Bigg\}\,\D x\,\D y,
	\end{align}
	where \begin{align}
	A_1:=&\{(\ve_1,\ve_2),(-\ve_1,\ve_2-\ve_1),(\ve_1-\ve_2,-\ve_2\},\notag\\A_2:=&\{(\vzero,\ve_2),(\vzero,\ve_1),(-\ve_2,-\ve_2),(-\ve_2,-\ve_2+\ve_1),\notag\\&(-\ve_1,-\ve_1),(-\ve_1,-\ve_1+\ve_2)\}\label{nearest set}.
	\end{align}
	Since $\vf_A$ is defined by $\vf$, $R_\text{g1}$ is the function of $\vf$. We define the $\Delta_A$:
	\begin{equation}
	\Delta_A:=-\lim_{\varepsilon\to 0}\sup_{\|\vf\|_{X_\varepsilon}=1}R_\text{g1}.\label{delta}
	\end{equation}
	
	This completes the proof.
\end{proof}

\begin{figure}[h]
		\centering
		\scalebox{0.2}{\includegraphics{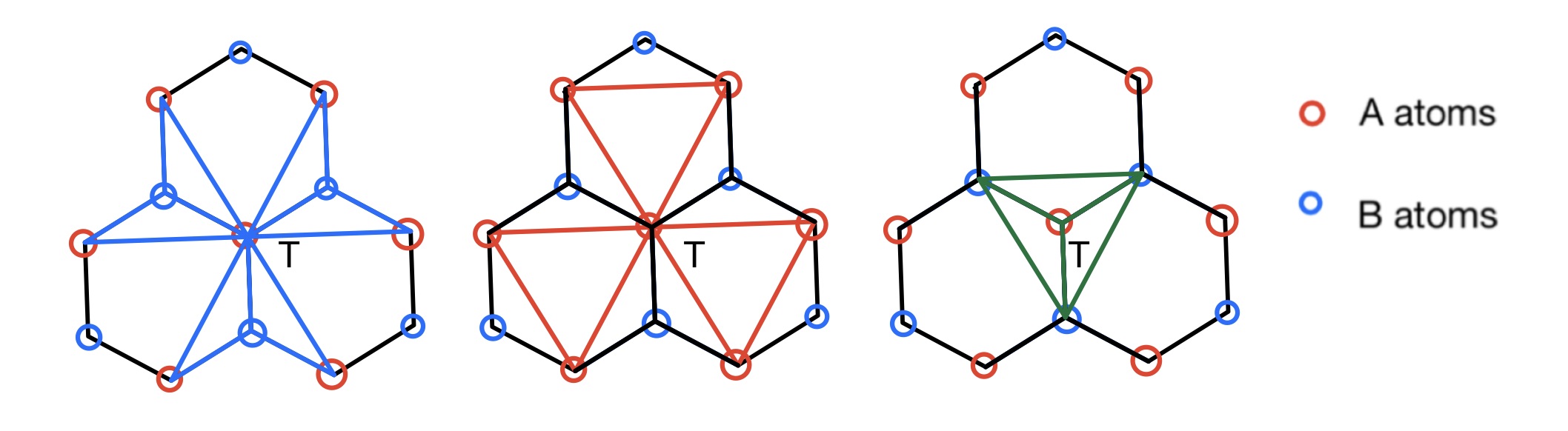}}
		\caption{Interpolation polynomial for nearest interaction.}
		\label{nearest} 
	\end{figure}

\begin{myre}\label{Ipni}
	In the \textbf{A8} and Lemma \ref{Delta}, we assume that $\Delta_A$ is small, it means that we assume the ``nearest" interactions dominate the elastic energy. Definition of the ``nearest" interactions on the complex lattice is illustrated in Figure~\ref{nearest}.
	For the $A$ atom, $T$, there are 12 triangles containing $T$: six blue triangles, three red triangles and three green triangles (come from B part). In Figure~\ref{nearest}, the blue triangles are for the  $AAB$ interactions, and the red triangles are for the $AAA$ interactions. In addition, in the part B, there are three $ABB$ interactions containing $T$. In other words, if we consider  those 12 interactions as nearest interactions in elastic energy, $\Delta_A=0$ can be obtained. Hence we define it as \textbf{interpolation polynomial for nearest interaction}.
\end{myre}

\begin{mylem}\label{norm di}
	Suppose that A1--A6 hold, $f\in X_{\varepsilon}$. There exist constants $\varepsilon_0$ and $C$, such that for $0<\varepsilon<\varepsilon_0$,
	\begin{equation}
	(1- \Delta_A-C\varepsilon)\|\vf\|^2_{X_\varepsilon}\le \|\vf_A\|^2_{X_0}\le(1+\Delta_A+C\varepsilon)\|\vf\|^2_{X_\varepsilon},
	\end{equation}
	where the $\vf_A$ defined in Lemma \ref{well define}.
\end{mylem}
\begin{proof}
	\begin{align}
	&\Big|\|\vf_A\|^2_{X_0} - \|\vf\|^2_{X_\varepsilon} \Big |=\left|R_\text{g1}+\frac{4\sqrt{3}}{3}\|\vf_A^{\perp}\|_{\vv}^2-\|\vf^\perp\|^2_\varepsilon\right|\notag\\\le&\big|-\Delta_A+C\varepsilon\big|\|\vf\|^2_{X_\varepsilon},
	\end{align}
	where the last inequality is due to Lemma \ref{Delta}.
\end{proof}

\begin{mypro}\label{sta a}
	Suppose that A1--A8 hold, and $\vv$ is the dislocation solution of the PN model in Theorem \ref{Theorem 1}. There exist constants $\varepsilon_0$ and $C>0$, such that for $0<\varepsilon<\varepsilon_0$ and $\vf\in X_{\varepsilon}$, we have
	$$\left\langle\delta^{2} E_\text{a}[\vv] \vf, \vf\right\rangle_{\varepsilon}\ge C\|\vf\|^2_{X_\varepsilon}.$$
\end{mypro}
\begin{proof}
	By Lemmas \ref{change},  \ref{Delta}, \ref{norm di} and Proposition \ref{stability of PN}, we have \begin{align}
	&\left\langle\delta^{2} E^A_\text{a}[\vv] \vf, \vf\right\rangle_{\varepsilon}\ge \left\langle\delta^{2} E_{\text{PN}}^{A}[\vv] \vf_A, \vf_A\right\rangle_0-\Delta_A\|\vf\|^2_{X_\varepsilon}+C\varepsilon\|\vf\|^2_{X_\varepsilon}\notag\\\ge &\bar{\vartheta} \|\vf_A\|^2_{X_0}-\Delta_A\|\vf\|^2_{X_\varepsilon}+C\varepsilon\|\vf\|^2_{X_\varepsilon}\ge\big[\bar{\vartheta}(1-\Delta_A)-\Delta_A-C\varepsilon\big]\|\vf\|^2_{X_\varepsilon}\notag\\\ge& \frac{1}{4}\bar{\vartheta}\|\vf\|^2_{X_\varepsilon},
	\end{align}
	where the last inequality is due to A6. Similar property can be obtained in $\delta^{2} E^B_\text{a}[\vv]$ part. Therefore we obtain \begin{equation}
	\left\langle\delta^{2} E_\text{a}[\vv] \vf, \vf\right\rangle_{\varepsilon}=\left\langle\delta^{2} E^A_\text{a}[\vv] \vf, \vf\right\rangle_{\varepsilon}+\left\langle\delta^{2} E^B_\text{a}[\vv] \vf, \vf\right\rangle_{\varepsilon}\ge\frac{1}{2}\bar{\vartheta}\|\vf\|^2_{X_\varepsilon}\ge\frac{1}{6}\vartheta\|\vf\|^2_{X_\varepsilon}.
	\end{equation}
\end{proof}

Now we have consistency of PN model (Proposition \ref{consistency pro}) and stability of atomistic model (Proposition \ref{sta a}), we can proof Theorem \ref{thm:Atom} similarly as in \cite{luo2018atomistic}.

\begin{mylem}\label{near small}
	Suppose that Assumptions A1--A6 hold. Let $\vv$ be the dislocation solution in PN model. There exist constants $\varepsilon_0$ and $C$, such that for $0<\varepsilon<\varepsilon_0$ and $\vu,\vu'\in S_{\varepsilon}$ satisfying $\|\vu-\vv\|_{X_{\varepsilon}}\le \varepsilon^2$  and $\|\vu'-\vv\|_{X_{\varepsilon}}\le \varepsilon^2$, we have
	\begin{equation}
	\left|\left\langle\left(\delta^{2}E_\text{PN} [\vu]-\delta^{2} E_\text{a}\left[\vu^{\prime}\right]\right) \vf, \vg\right\rangle_{\varepsilon}\right| \leq C\varepsilon^{-1} \left\|\vu-\vu^{\prime}\right\| _{X_{\varepsilon}}\|\vf\|_{X_{\varepsilon}}\|\vg\|_{X_{\varepsilon}}
	\label{lemm 9}
	\end{equation}
	for all $\vf,\vg\in X_\varepsilon$ and $C$ is independent on $\varepsilon$.
\end{mylem}
\begin{proof}
	
By definition of the norm in $X_\varepsilon$, we have for all $\vx\in \mathbb{L}$: \[|D_{\vs_1}(\vu_{\vx}^{A\pm}-\vv_{\vx}^{A\pm})|=\varepsilon^{-1}|D_{\vs_1}\varepsilon( \vu_{\vx}^{A\pm}-\vv_{\vx}^{A\pm})|\le\varepsilon^{-1}\|\vu-\vv\|_{X_\varepsilon}.\]Therefore we can get$$
	\left\|D_{\vs_1}(\vu-\vv)\right\|_{L_{\varepsilon}^{\infty}} \leq\varepsilon^{-1 }\|\vu-\vv\|_{X_{\varepsilon}} \leq 1.
	$$ Furthermore, we can get $\|D_{\vs_1}\vu\|_{L_{\varepsilon}^{\infty}}\le C$ since $\|D_{\vs_1}\vv\|_{L_{\varepsilon}^{\infty}}\le \|\nabla \vv\|_\infty\le C$. Similarly, $$\|D_{\vs_1}(\vu-\vu')\|_{L_{\varepsilon}^{\infty}}\le \varepsilon^{-1 }\|\vu-\vu'\|_{X_{\varepsilon}}\le 1.$$  Similar property can be obtain in $D^{\pm \vp}_{\vs_1}$.
	
	As for operator $V$, we can have
	\begin{align}
	&\sum_{\vs}\big|\partial_{ij}V(\vs_1+\varepsilon D_{\vs_1}\vu_{\vx}^{A\pm},\vs_2+\varepsilon D_{\vs_2}\vu_{\vx}^{A\pm})\notag\\&-\partial_{ij}V(\vs_1+\varepsilon D_{\vs_1}(\vu')_{\vx}^{A\pm},\vs_2+\varepsilon D_{\vs_2}(\vu')_{\vx}^{A\pm})\big|\notag\\\le&\sum_{\vs}(V_{(i,j+1),\vs_1,\vs_2}+V_{(i+1,j),\vs_1,\vs_2})\varepsilon\big[\left|(D_{\vs_2}+D_{\vs_1})\left(\vu_{\vx}^{A\pm}-(\vu')_{\vx}^{A\pm}\right)\right|\big] \notag\\\leq&C \left\|\vu^{\prime}-\vu\right\|_{X_{\varepsilon}},
	\end{align}
	where $\varepsilon$ is sufficiently small.
	Hence, we obtain \begin{align}
	\left|\left\langle\left(\delta^{2}E_\text{elas} [\vu]-\delta^{2} E_\text{a-el}\left[\vu^{\prime}\right]\right) \vf, \vg\right\rangle_{\varepsilon}\right|&\le C\left\|\vu^{\prime}-\vu\right\|_{X_{\varepsilon}}(\vf,\vg)_{X_\varepsilon}\notag\\&\le C \left\|\vu-\vu^{\prime}\right\| _{X_{\varepsilon}}\|\vf\|_{X_{\varepsilon}}\|\vg\|_{X_{\varepsilon}}.
	\end{align}

	For the $\gamma$-surface, we use this notation  in the proof:
	$${U}_{k,\vs,h}=\sup_{|\vxi-\vs+\frac{1}{2}| \leq h}|U^{(k)}(\vxi)|,$$where $k=0,1,2,3$. We can prove the same property of ${U}_{k,\vs,h}$ as in the proof of Lemma \ref{A4 come}:
	\begin{equation}
	\text{For all } h>0,~~ \text{there exists } H>0,\text{ such that }\sum_{\vs}{U}_{k,\vs,h}|\vs|^{k+1}\le H.\label{uk}
	\end{equation}
	Since $\vu,\vu'\in S_\varepsilon$, we can get $\|\vu\|_{\infty}<\infty,\|\vu'\|_{\infty}<\infty$. Thus we can denote $r=\|\vu\|_{\infty}+\|\vu'\|_{\infty}$ and get
	\begin{align}&\sum_{\vs}\left|U^{\prime \prime}\left(\vs+\vu_{\vx+\vs_2}^{A+}-\vu_{\vx}^{A-}-\vd\right)-U^{\prime \prime}\left(\vs+(\vu')_{\vx+\vs_2}^{A+}-(\vu')_{\vx}^{A-}-\vd\right)\right| \notag\\  \leq& \sum_{\vs}U_{3,\vs,r}\left|\varepsilon^{-1} D_{\vs_2}\left (\vu_{\vx}^{A+}-\vu_{\vx}^{\prime A+}\right)+\left(\vu_{\vx}^{\perp A+}-\vu_{\vx}^{\prime\perp A+}\right)\right| \notag\\ \leq& C \varepsilon^{-1}\left\|\vu^{\prime}-\vu\right\| _{X_{\varepsilon}}.
	\end{align}
	
	Therefore, we get \begin{align}
	\left|\left\langle\left(\delta^{2}E_\text{mis} [\vu]-\delta^{2} E_\text{a-mis}\left[\vu^{\prime}\right]\right) \vf, \vg\right\rangle_{\varepsilon}\right|&\le C\varepsilon^{-1}\left\|\vu^{\prime}-\vu\right\|_{X_{\varepsilon}}(\vf,\vg)_{X_\varepsilon}\notag\\&\le C \varepsilon^{-1}\left\|\vu-\vu^{\prime}\right\| _{X_{\varepsilon}}\|\vf\|_{X_{\varepsilon}}\|\vg\|_{X_{\varepsilon}}.
	\end{align}

	Finally, we can get Lemma \ref{near small} by direct calculation.
\end{proof}

\begin{mypro}\label{stta} Suppose that Assumptions A1--A8 hold. Let $\vv$ be the dislocation solution in PN model. There exist constants $\varepsilon_0$ and $C$, such that for $0<\varepsilon<\varepsilon_0$ and $\vu\in S_{\varepsilon}$ satisfying $\|\vu-\vv\|_{X_{\varepsilon}}\le \varepsilon^2$, we have for $\vf\in X_{\varepsilon} $
	\begin{equation}
	\left\langle\delta^{2} E_\text{a}[\vu] \vf, \vf\right\rangle_{\varepsilon}\ge C\|\vf\|^2_{X_\varepsilon},\label{prop 5}
	\end{equation}where $C$ is independent of $\varepsilon$.
\end{mypro}
\begin{proof} Let $\vu'=\vv$ and $\vg=\vf$ in Lemma \ref{near small}, we have $$\left|\left\langle\left(\delta^{2} E_\text{a}[\vu]-\delta^{2} E_\text{a}\left[\vv\right]\right) \vf, \vf\right\rangle_{\varepsilon}\right| \leq C \varepsilon\|\vf\|^2_{X_{\varepsilon}}.$$ By Proposition \ref{sta a}, we have $\left\langle\delta^{2} E_\text{a}[\vu] \vf, \vf\right\rangle_{\varepsilon}\ge C\|\vf\|^2_{X_\varepsilon}$  for a sufficient small $\varepsilon$. \end{proof}

\subsection{Proof of Theorem \ref{thm:Atom}}.

\begin{proof}[Proof of Theorem 2]
	First, we define a close subset $B$ of $X_\varepsilon$ with an uncertain constant $C_B$, which will be defined later :
	\begin{equation}
	B=\left\{\vpsi \in X_{\varepsilon} :\|\vpsi\|_{X_{\varepsilon}} \leq C_{B} \varepsilon^2\right\}.\label{ball}
	\end{equation}
	We define the operator $A_{\vpsi} : X_\varepsilon\to X_\varepsilon$ following:
	\begin{align}
	\left(A_{\vpsi} \vf, \vg\right)_{X_{\varepsilon}}&=\int_{0}^{1}\left\langle\delta^{2} E_{\mathrm{a}}\left[\vu^{t}\right] \vf, \vg\right\rangle_{\varepsilon} \mathrm{d} t, \quad \vf, \vg \in X_{\varepsilon},
	\label{aw}
	\end{align}
	where $\vu^t=\vv+t\vpsi$ for $t\in [0,1]$. Next, we prove that for every $\vpsi$, $A_{\vpsi}$ is invertible. Thanks to Lax--Milgram Theorem, we only need to prove the following three parts:
	
	\textbf{(i)} $A_{\vpsi}$ is well-defined and $\left(A_{\vpsi} \vf, \vg\right)_{X_{\varepsilon}}$ is bi-linear.
	
	\textbf{(ii)} There exists a constant $c>0$ such that for every $\vf, \vg \in X_{\varepsilon}$, $$\left(A_{\vpsi} \vf, \vg\right)_{X_{\varepsilon}}\le c\|\vf\|_{X_\varepsilon}\|\vg\|_{X_\varepsilon}.$$
	
	\textbf{(iii)} There exists a constant $C>0$ such that for every $\vf\in X_{\varepsilon}$, $$\left(A_{\vpsi} \vf, \vf\right)_{X_{\varepsilon}}\ge C\|\vf\|_{X_\varepsilon}^2.$$
	
	For (i) and (ii), it is easy to check from the definition of $A_{\vpsi}$. For (iii), for every $t$, we have $\|\vu^t-\vv\|_{X_\varepsilon}\le t\|\vpsi\|_{X_\varepsilon}\le C_B\varepsilon^2$. Hence for Proposition \ref{stta}, we have $
	\left\langle\delta^{2} E_{\mathrm{a}}\left[\vu^{t}\right] \vf, \vf\right\rangle_{\varepsilon} \geq C\|\vf\|_{X_{\varepsilon}}^{2}
	.$ Thus we can get $\left(A_{\vpsi} \vf, \vf\right)_{X_{\varepsilon}}\ge C\|\vf\|_{X_\varepsilon}^2.$ By the above analysis, $A_{\vpsi}$ is invertible.
	
	By the Newton--Leibniz formula, we have $$
	\begin{aligned}\left\langle\delta E_{\mathrm{a}}[\vv+\vw], \vf\right\rangle_{\varepsilon} -\left\langle\delta E_{\mathrm{a}}[\vv], \vf\right\rangle_{\varepsilon}&=\int_{0}^{1}\left\langle\delta^{2} E_{\mathrm{a}}\left[\vu^{t}\right] \vw, \vf\right\rangle_{\varepsilon} \mathrm{d} t \\ &=\left(A_{\vw}\vw , \vf\right)_{X_{\varepsilon}} ,\end{aligned}
	$$where $\vw\in B$.
	
	If $\vv+\vw$ solve atomistic model, we can get a $\vw$ to solve
	\begin{equation}
	\left(A_{\vw}\vw, \vf\right)_{X_{\varepsilon}}=-\left\langle\delta E_{\mathrm{a}}[\vv], \vf\right\rangle_{\varepsilon}\text{  for all  }\vf\in X_\varepsilon.\label{solution}
	\end{equation}
	Next we define a map $G:B\to X_\varepsilon$ for $\vw\in B$ as
	\begin{equation}
	\left(A_{\vw}G(\vw), \vf\right)_{X_{\varepsilon}}=-\left\langle\delta E_{\mathrm{a}}[\vv], \vf\right\rangle_{\varepsilon}\text{  for all  }\vf\in X_\varepsilon.\label{main}
	\end{equation}Now we prove that $G$ is a contraction mapping.
	
	\textbf{(i)} Find a properly $C_B$ in the definition of $B$ to make $G(B)\subset B$. Since $\left(A_{\vpsi} \vf, \vf\right)_{X_{\varepsilon}}\ge C\|\vf\|_{X_\varepsilon}^2$ and by Proposition \ref{consistency pro}, we have
	\begin{align} C\|G(\vw)\|_{X_{\varepsilon}}^{2} & \leq\left(A_{\vw} G(\vw), G(\vw)\right)_{X_{\varepsilon}} \notag\\ & \leq\left|\left\langle\delta E_{\mathrm{a}}[\vv], G(\vw)\right\rangle_{\varepsilon}\right| \notag\\ & \leq \varepsilon^2\|G(\vw)\|_{X_{\varepsilon}}. \label{constant}\end{align}
	Thus we can choose a proper $C_B$ from \eqref{constant}.
	
	\textbf{(ii)} Show that
	$
	\left\|G(\vw)-G\left(\vw^{\prime}\right)\right\|_{X_{\varepsilon}} \leq L\left\|\vw-\vw^{\prime}\right\|_{X_{\varepsilon}} \text { for any }\vw',\vw\in B
	$ and $L<1$.
	
	Here we have $$
	\begin{aligned}\left\|G(\vw)-G\left(\vw^{\prime}\right)\right\|_{X_{\varepsilon}}^{2} &=\left(G(\vw)-G\left(\vw^{\prime}\right),G(\vw)-G\left(\vw^{\prime}\right)\right)_{X_{\varepsilon}}\\&=\left|\left\langle\delta E_{\mathrm{a}}[\vv],\left(A_{\vw}^{-1}-A_{\vw^{\prime}}^{-1}\right)\left(G(\vw)-G\left(\vw^{\prime}\right)\right)\right\rangle_{\varepsilon}\right| \\ &=C\varepsilon^2\left\|\left(A_{\vw}^{-1}-A_{\vw^{\prime}}^{-1}\right)\left(G(\vw)-G\left(\vw^{\prime}\right)\right)\right\|_{X_{\varepsilon}} \\ &=C\varepsilon^2\left\|A_{\vw}^{-1}\left(A_{\vw}-A_{\vw^{\prime}}\right) A_{\vw^{\prime}}^{-1}\left(G(\vw)-G\left(\vw^{\prime}\right)\right)\right\|_{X_{\varepsilon}} \\ & \leq C\varepsilon^2\left\|A_{\vw}^{-1}\right\|_{\mathrm{op}} \cdot\left\|A_{\vw}-A_{\vw^{\prime}}\right\|_{\mathrm{op}} \cdot\left\|A_{\vw^{\prime}}^{-1}\right\|_{\mathrm{op}} \\ & ~~\cdot\left\|G(\vw)-G\left(\vw^{\prime}\right)\right\|_{X_{\varepsilon}}, \end{aligned}
	$$where $\left\| \cdot\right\|_{\mathrm{op}}$ is norm of operator.
	
	Thanks to $$
	C\left\|A_{\vw}^{-1} \vf\right\|_{X_{\varepsilon}} \leq \frac{( A_{\vw} A_{\vw}^{-1} \vf, A_{\vw}^{-1} \vf)_{X_\varepsilon}}{\left\|A_{\vw}^{-1} \vf\right\|_{X_{\varepsilon}}} \leq\|\vf\|_{X_{\varepsilon}},
	$$ we have
	\begin{equation}
	\left\|A_{\vw}^{-1}\right\|_{\mathrm{op}} \leq C^{-1},\left\|A_{\vw^{\prime}}^{-1}\right\|_{\mathrm{op}} \leq C^{-1}
	.\label{op 1}
	\end{equation}
	By Lemma \ref{near small}, we have $$
	\begin{aligned}&\left\|\left(A_{\vw}-A_{\vw^{\prime}}\right) \vf\right\|_{X_{\varepsilon}}^{2}\\  =&\int_{0}^{1}\left\langle\left(\delta^{2} E_{\mathrm{a}}[\vv+t \vw]-\delta^{2} E_{\mathrm{a}}\left[\vv+t \vw^{\prime}\right]\right) \vf,\left(A_{\vw}-A_{\vw^{\prime}}\right) \vf\right\rangle_{\varepsilon} \mathrm{d} t \\ \leq&  \int_{0}^{1} C \varepsilon^{-1 }\left\|t \vw-t \vw^{\prime}\right\|_{X_{\varepsilon}}\|\vf\|_{X_{\varepsilon}}\left\|\left(A_{\vw}-A_{\vw^{\prime}}\right) \vf\right\|_{X_{\varepsilon}} \mathrm{d} t \\ \leq&  C \varepsilon^{-1 }\left\|\vw-\vw^{\prime}\right\|_{X_{\varepsilon}}\|\vf\|_{X_{\varepsilon}}\left\|\left(A_{\vw}-A_{\vw^{\prime}}\right) \vf\right\|_{X_{\varepsilon}} \end{aligned}
	.$$Therefore
	\begin{equation}
	\left\|A_{\vw}-A_{\vw^{\prime}}\right\|_{\mathrm{op}} \leq C \varepsilon^{-1 }\left\|\vw-\vw^{\prime}\right\|_{X_{\varepsilon}}
	\label{op 2}
	\end{equation}
	Thanks to \eqref{op 1} and \eqref{op 2}, we can get :$$
	\left\|G(\vw)-G\left(\vw^{\prime}\right)\right\|_{X_{\varepsilon}} \leq C \left\|\vw-\vw^{\prime}\right\|_{X_{\varepsilon}} C \varepsilon \leq L\left\|\vw-\vw^{\prime}\right\|_{X_{\varepsilon}},
	$$where $L<1$ for sufficient small $\varepsilon$.
	
	Thus $G$ is a contraction mapping, which means there is a unique fixed point $\vw_\varepsilon$, i.e., $G(\vw^\varepsilon)=\vw^\varepsilon$. This implies that $\vv^\varepsilon:=\vv+\vw^\varepsilon$ is the solution of atomistic model. Furthermore, by Proposition \ref{sta a}, we have that for all  $f\in X_\varepsilon$, \[\begin{aligned}\left\langle\delta^{2} E_\text{a}[\vv^\varepsilon] \vf, \vf\right\rangle_{\varepsilon}&\ge C\|\vf\|^2_{X_\varepsilon},\\\left\langle\delta E_\text{a}[\vv^\varepsilon], \vf\right\rangle_{\varepsilon}&=0.\end{aligned}\]
	Therefore $\vv^\varepsilon$ is $X_\varepsilon$--local minimizer of the energy.
\end{proof}

\section*{Acknowledgements}
This work was supported by the Hong Kong Research Grants Council General Research Fund 16313316.

\bibliographystyle{plain}
\bibliography{references}
\end{document}